\documentclass[11pt, reqno]{amsart}

\usepackage{graphicx}
\usepackage[normalem]{ulem}
\usepackage{mathrsfs}
\usepackage{amssymb}
\usepackage{amsfonts, mathtools}
\usepackage{amssymb,amsmath,amsthm,tikz-cd, tikz}
\usepackage{stix}

\tikzcdset{arrow style=tikz, diagrams={>=stealth'}}

\usepackage{hyperref}
\usepackage{wrapfig}
\usepackage{subcaption}

\usetikzlibrary{patterns, matrix,shapes,arrows,calc,topaths,intersections,positioning,decorations.pathreplacing,decorations.pathmorphing,fit, 3d}

\newif\ifpaper
\papertrue

\usepackage{color}

\title {GKZ discriminant and Multiplicities}
\author{Jesse Huang and Peng Zhou}
\date{\today}

\usepackage{color}

\DeclareMathOperator{\codim}{codim}

\DeclareMathOperator{\coker}{coker}
\DeclareMathOperator{\cone}{cone}

\DeclareMathOperator{\conv}{conv}

\DeclareMathOperator{\Coh}{Coh}

\DeclareMathOperator{\Ext}{Ext}

\DeclareMathOperator{\Fuk}{Fuk}

\DeclareMathOperator{\Hom}{Hom}

\DeclareMathOperator{\Log}{Log}

\DeclareMathOperator{\rank}{rank}

\renewcommand{\span}{\z{span}}

\DeclareMathOperator{\vol}{vol}

\newcommand{\NN}{\mathsf{N}}
\newcommand{\MM}{\mathsf{M}}

\newcommand{\bSigma}{ {\bf \Sigma} }
\newcommand{\bsigma}{ {\hat \sigma} }

\newcommand{\Zv}{{\mathbb{Z}^\vee}}

\newcommand{\ZvN}{{(\Zv)^N}}

\newcommand{\Lv}{L^\vee}
\newcommand{\Ga}{\Gamma}
\newcommand{\qGa}{{/\Gamma}}

\newcommand{\z}{\text}

\newcommand{\la}{\langle}
\newcommand{\ra}{\rangle}

\newcommand{\wt}{\widetilde}

\newcommand{\h}{\hat}

\newcommand{\fcal}{\mathcal{F}}

\newcommand{\wcal}{\mathcal{W}}

\newcommand{\A}{\mathbb{A}}

\newcommand{\C}{\mathbb{C}}

\newcommand{\R}{\mathbb{R}}

\newcommand{\Z}{\mathbb{Z}}

\newcommand{\Q}{\mathbb{Q}}

\newcommand{\F}{\mathbb{F}}
\newcommand{\T}{\mathbb{T}}
\renewcommand{\P}{\mathbb{P}}

\newcommand{\xto}{\xrightarrow}

\newcommand{\RM}{\backslash}

\newcommand{\into}{\hookrightarrow}

\newcommand{\onto}{\twoheadrightarrow}

\newcommand{\bea}{\begin{eqnarray*} }
\newcommand{\eea}{\end{eqnarray*} }
\newcommand{\be}{\begin{equation} }
\newcommand{\ee}{\end{equation} }
\newcommand{\bp}{\begin{proposition}}
\newcommand{\ep}{\end{proposition}}
\newcommand{\bt}{\begin{theo}}
\newcommand{\et}{\end{theo}}

\newcommand{\btu}{\begin{theou}}
\newcommand{\etu}{\end{theou}}
\newcommand{\bpf}{\begin{proof}}
\newcommand{\epf}{\end{proof}}
\newcommand{\bl}{\begin{lemma}}
\newcommand{\el}{\end{lemma}}
\newcommand{\bc}{\begin{corollary}}
\newcommand{\ec}{\end{corollary}}
\newcommand{\bd}{\begin{definition}}
\newcommand{\ed}{\end{definition}}
\newcommand{\bex}{\begin{example}}
\newcommand{\eex}{\end{example}}
\newcommand{\bA}{\left(\begin{array}}
\newcommand{\eA}{\end{array}\right)}
\newcommand{\bma}{\begin{bmatrix}}
\newcommand{\ema}{\end{bmatrix}}
\newcommand{\bcd}{\begin{tikzcd}}
\newcommand{\ecd}{\end{tikzcd}}
\newcommand{\bcs}{\begin{cases}}
\newcommand{\ecs}{\end{cases}}
\newcommand{\bee}{\begin{eqnarray} }
\newcommand{\eee}{\end{eqnarray} }
\newcommand{\brem}{\begin{remark}}
\newcommand{\erem}{\end{remark}}
\newcommand{\bnum}{\begin{enumerate}}
\newcommand{\enum}{\end{enumerate}}

\newtheorem*{theou}{Main Theorem}
\newtheorem{theo}{Theorem}[section]
\newtheorem{lemma}[theo]{Lemma}
\newtheorem{corollary}[theo]{Corollary}
\newtheorem{proposition}[theo]{Proposition}

\newtheorem{definition}[theo]{Definition}
\newtheorem{remark}[theo]{Remark}

\newtheorem{examplex}[theo]{Example}
\newenvironment{example}
  {\pushQED{\qed}\examplex}
  {\popQED\endexamplex}

\theoremstyle{plain}
\numberwithin{equation}{section}

\newcommand{\CS}{{\C^*}}

\newcommand{\In}{\subset}

\renewcommand{\cong}{\simeq}

\newcommand*{\hrlen}{5}
\newcommand*{\hramp}{3}

\tikzset{
asdstyle/.style={blue,thick},
righthairs/.style={postaction={decorate,draw,decoration={border,amplitude=\hramp,segment length=\hrlen,angle=-90,pre=moveto,pre length=\hrlen/2}}},
lefthairs/.style={postaction={decorate,draw,decoration={border,amplitude=\hramp,segment length=\hrlen,angle=90,pre=moveto,pre length=\hrlen/2}}},
righthairsnogap/.style={postaction={decorate,draw,decoration={border,amplitude=\hramp,segment length=\hrlen,angle=-90}}},
lefthairsnogap/.style={postaction={decorate,draw,decoration={border,amplitude=\hramp,segment length=\hrlen,angle=90}}},
graphstyle/.style={thick},
arrowstyle/.style={thick,decorate,decoration={snake,amplitude=1.7,segment length=10pt,post length=.5mm,pre length=0}},
genmapstyle/.style={thick,-stealth'},
arrhdstyle/.style={thick},
exceptarcstyle/.style={red, ultra thick},
dualquiverstyle/.style={thick,->},
patstyle/.style={pattern color = gray, pattern = north east lines, opacity=0.3}
}

\begin{document}

\begin{abstract}
Let $T=(\C^*)^k$ act on $V=\C^N$ faithfully and preserving the volume form, i.e.  $(\C^*)^k \into \text{SL}(V)$. On the B-side, we have toric stacks $Z_W$ (see Eq. \ref{eq:ZW})labelled by walls $W$ in the GKZ fan, and $Z_{/F}$ labelled by faces of a polytope corresponding to minimal semi-orthogonal decomposition (SOD) components. The B-side multiplicity $n^B_{W,F}$, well-defined by a result of Kite-Segal \cite{kite-segal}, is the number of times  $\Coh(Z_{/F})$ appears in a complete SOD of $\Coh(Z_W)$.  On the A-side, we have the GKZ discriminant loci components $\nabla_F \In (\C^*)^k$, and its tropicalization $\nabla^{trop}_{F} \In \R^k$. The A-side multiplicity $n^A_{W, F}$ is defined as the multiplicity of the tropical complex $\nabla^{trop}_{F}$ on wall $W$.  We prove that $n^A_{W,F} = n^B_{W,F}$, confirming a conjecture in  Kite-Segal \cite{kite-segal} inspired by \cite{aspinwall2017mirror}. Our proof is based on the result of Horja-Katzarkov \cite{horja2022discriminants} and a lemma about B-side SOD multiplicity, which allows us to reduce to lower dimension just as in A-side \cite{GKZ-book}[Ch 11].

%$Z_W = [V^{\lambda_W} //_W (T/\lambda_W)]$. 

\end{abstract}

\maketitle

\section{Introduction}
Homological mirror symmetry (HMS) for toric varieties is a well-studied subject, yet it can still offer new insights to classical problems. Our paper concerns a numerical conjecture that is a shadow of a full categorical conjecture \cite{aspinwall2017mirror, kite-segal, horja2022discriminants}. We first sketch the categorical conjecture. On the B-side,  we have a toric Calabi-Yau GIT problem and we study the derived equivalences and semi-orthogonal decompositions that arise from wall-crossing. On the A-side we have a fiberwise partially wrapped Fukaya category \cite{abouzaid-auroux2111homological} associated to a fibration $\pi: Y \to B$ and a superpotential $W: Y \to \C$, and we want to study the ``pushforward'' of $\Fuk(Y, W)$
along $\pi$ to get a Fukaya category on $B$ with categorical coefficient. For $b \in B$, let $Y_b$ be the fiber over $b$ and $W_b$ the restriction of $W$, then there is a discriminant loci $\nabla \In B$ where the fiberwise wrapped Fukaya category $\Fuk(Y_b, W_b)$ become ``degenerate''. 
The full HMS predicts that, for each wall crossing $W$ (corresponding to certain asymptotic region of $\nabla$), there is a B-side perverse schober coming from SOD of $\Coh(Z_W)$ where $Z_W$ is certain toric stack associated to $W$ (see Eq. \eqref{eq:ZW}), and there is an A-side analog coming from a transversal curve (annuli) intersecting the discriminant loci $\nabla$ in the asymptotic region of the wall. Our main theorem is a verification that the two schobers has the same number of singularities of each type.

The full conjecture is explained beautifully in the Kite-Segal paper \cite{kite-segal}, which is based on the physicists conjecture \cite{aspinwall2017mirror}. In the remaining part of the introduction, we will focus on the numerical conjecture and state our main results.

\subsection{Main Result}

Our input data is a collection of lattice points $q_1, \cdots, q_N$ in $\Z^k$, where $q_i=(q_{i1},\cdots,q_{ik})$, called weights. We assume that $q_i$ generate the linear space $\R^k$, and they satisfy the balanced condition $\sum_i q_i=0$. Equivalently, the input data is a full rank linear map 
$$ Q: \Z^N \to \Z^k, \quad e_i \mapsto q_i,  $$
such that $Q(1,\cdots,1)=0$. 
From this input data, we can set up two problems (called A-side and B-side) as follows. 

First, we consider the dual of $Q$, $Q^\vee: \Z^k \to \Z^N$, and let $A$ denote the cokernel map
$$ A: \Z^N \to \NN $$
where $\NN$ may have torsion if $Q$ is not surjective. 

For the sake of the introduction, we assume that $Q$ is surjective and thus $\NN$ is a lattice. We further assume that for any $i \in [N]=\{1,\cdots,N\}$,  $a_i=A(e_i)$ are distinct. This is for easy quotation of results in \cite{GKZ-book}. We also abuse notation and let $A$ denote the set $\{a_i\}$. Both of the assumptions can be easily removed, as described in the appendix. 

\subsubsection{The A-side setup}
The A-side problem concerns a holomorphic function $W: Y \to \C$ and a fibration $\pi: Y \to B$. Here $Y=(\C^*)^N, B=(\C^*)^k$, and 
$$ W: (\C^*)^N \to \C, \quad W(z) = z_1 + \cdots + z_N, $$
$$ \pi=Q_\CS: (\C^*)^N \to (\C^*)^k, \quad Q_\CS(z) = (\prod_{i=1}^N z_i^{q_{i1}}, \cdots, \prod_{i=1}^N z_i^{q_{ik}}). $$

Gelfand-Kapranov-Zelevinsky \cite{GKZ-book} defined a polynomial $E_A$ of $N$ variables, called principal $A$-determinant, whose vanishing loci is a variety $\wt \nabla_{GKZ} \In (\C^*)^N$. This variety is the preimage, under $\pi$, of a variety $\nabla_{GKZ} \In B$. 

Let $\Pi = \conv(\{0\} \cup A)$. \footnote{Our definition of $\Pi$ is different from \cite{kite-segal}, which defines $\Pi:=\conv(A)$. We use this convention to better describe multiplicities using polytopes. See Section \ref{ss: recursive}. }
%By the Calabi-Yau condition $\sum_i q_i=0$, $\Pi$ is a $k$ dimensional polytope in an affine $\Z^{k-1}$. 
Let $\fcal_0$ denote the faces of $\Pi$ that contains $0$ ($\Pi$ itself is a face), and let $\fcal \In \fcal_0$ denote the set of ``minimal faces'', where a $F$ is minimal if one remove any point in $F \cap A$ then the remaining points still generate the span of $F$. (c.f. Definition \ref{d:minface-relsub})Note that $\{0\}$ is always a minimal face, since $\{0\} \cap A = \emptyset$. 

For each minimal face $F$, there is an irreducible variety $\nabla_F \In B$ 
%{\red put forward reference where is this defined}
, and the divisor $\nabla_{GKZ}$ can be decomposed as
$$ \nabla_{GKZ} = \sum_{F \in \fcal} m_F \cdot \nabla_F, $$
where $m_F$ is some integer (not to be confused with our A-side multiplicity). In \cite{aspinwall2017mirror}, the multiplicity $m_F$ is interpreted as the rank of $K_0$ of a certain Higgs problem (Section \ref{ss:Coulomb}). %Even better, if $F \neq \Pi$, then $\nabla_F$ is the pullback along a certain torus quotient $\pi_F: B \to B_F$ of a discrimiant $\nabla'_F \In B_F$ for a problem of lower rank. 

The tropicalization $\nabla^{trop}_{GKZ} \In \R^k$ of $\nabla_{GKZ} \In (\C^*)^k$ is the codimension-1 part of the GKZ-fan. Given a subvariety $V \In (\C^*)^k$, we define its tropicalization as follows. Consider the map
$ \Log: (\C^*)^k \to \R^k$ by componentwise $z \mapsto \log|z|$. Then, we define $V^{trop} = \lim_{t \to 0^+} t \cdot \Log(V)$. As shown in \cite{GKZ-book}, the Newton polytope for the defining polynomial of $\nabla_{GKZ}$ is the secondary polytope, whose exterior normal fan is the GKZ fan $\Sigma_{GKZ} \In \R^k$. 

The tropicalization method \cite{mikhalkin2004decomposition} outputs not just a set, but a polyhedral complex with multiplicities. It allows for easy computation of intersection multiplicities. 

\begin{example}
See Figure \ref{fig:ex_trop}.  The tropicalization of $x+y=1$ in $(\C^*)^2$ is a tropical curve with weight $1$ on each leg, where the three legs represent the approximate equations $x=1, y=1, x+y=0$. 
The tropicalization of $x^2+y^3=1$ is a tropical curve consisting of three legs, representing the three region $x^2 \approx 1, y^3 \approx 1$ and $x^2+y^3 \approx 0$.  The weight $2$ over the leg $x^2 \approx 1$ means there are two branches of complex solutions $x=\pm 1$ above the tropical 'shadow'. 

\begin{figure}
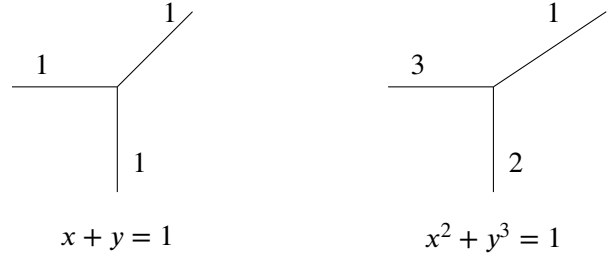

    \centering
    \tikz{ \begin{scope}
    \draw (0,0) -- (-1.4,0);
    \draw (0,0) -- (0,-1.4);
    \draw (0,0) -- (1,1);
    \node at (-1,0.3) {$1$};
    \node at (0.3,-1) {$1$};
    \node at (0.7,1) {$1$};
    \node at (0, -2) {$x+y=1$};
    \end{scope}
    \begin{scope}[shift={(5,0)}]
    \draw (0,0) -- (-1.4,0);
    \draw (0,0) -- (0,-1.4);
    \draw (0,0) -- (1.5,1);
    \node at (-1,0.3) {$3$};
    \node at (0.3,-1) {$2$};
    \node at (0.8,1) {$1$};
    \node at (0, -2) {$x^2+y^3=1$};
    \end{scope}
    }
    \caption{Tropicalization gives balanced polyhehral complex.  The integers labels multiplicities on each cell. \label{fig:ex_trop}}
    
\end{figure}
\end{example}

Let $W$ be a wall in the GKZ fan. We can ask for the multiplicity of a tropical complex $\nabla^{trop}_{GKZ}$ or $\nabla^{trop}_F$ along $W$, and denote them by $n^A_{W,GKZ}, n^A_{W,F}$. We get
$$ n^A_{W,GKZ} = \sum_{F \in \fcal} m_F \cdot n^A_{W,F}. $$

It is not too hard to compute $n^A_{W,GKZ}$ in terms of volumes of some polytope, but to compute $n^A_{W,F}$, some extra work is needed. The key observation, made in \cite{GKZ-book}[Ch 11], is that $\nabla_F$ actually comes from a lower dimensional problem, and one can solve for $n^A_{W,F}$ recursively (or express the result as an alternating sum, if $m_F = 1$).  We bring this idea to the B-side and show that \emph{the same} recursive relation also holds for SOD multiplicity, as predicted by mirror symmetry. 

\subsubsection{The B-side setup}
Let $T=(\C^*)^k$ acts on $\C^N$ with weights $q_1, \cdots, q_N$. There is a GKZ fan in $\R^k$ that labels all possible GIT quotients. The GKZ fan is the ``shadow'' of the $\R_{\geq 0}^N$ (the moment polytope of $\C^N$ for the $(\C^*)^N$-action) under the map $Q_\R: \R^N \to \R^k$. 

The GIT quotient stacks $X_C$ for different chambers $C \In \R^r$ are derived equivalent. For two adjacent chambers $C_1, C_2$ separated by a wall $W$, we have $\Z$ many equivalence functors 
$$ \phi_i = \phi_{i,W, C_1, C_2}: \Coh(X_{C_1}) \to \Coh(X_{C_2}), i \in \Z $$
where $i$ labels of a choice of window subcategory \cite{herbst2008phases, segal2011equivalences, halpern2015derived, BFK} in the GIT quotient $X_W = [V/\!/\!_\theta T]$, $\theta \in W$. Let $\lambda_W \in X_*(T)$ be a 1-parameter subgroup in $T$, such that $\lambda_W \perp \span_\R(W)$. Let $V^{\lambda_W}$ denote the fixed loci of $\lambda_W$, then $T / \lambda_W$ acts on $V^{\lambda_W}$. We define the GIT quotient stack 
\begin{equation}
    \label{eq:ZW}
    Z_W = [V^{\lambda_W} /\!/\!_W (T/\lambda_W)]
\end{equation} 
See also definition \ref{d:Z-variety}.

It is interesting to study the autoequivalence induced by window shift $\mu = \phi_{1}^{-1} \phi_0$ (for simplicity, we fix a window shift here). It comes from a spherical functor  
$$ S:  \Phi =\Coh(Z_W) \to \Psi = \Coh(X_{C_1}) $$
in that $\mu = \cone(1 \to S S^l)[-1]$, where $\Phi$ is called the vanishing cycle category, and $\Psi$ the nearby cycle category. Hence, this defines a B-model perverse schober over a disk.

Halpern-Leistner and Shipman \cite{halpern2016autoequivalences} showed that if $\Phi$ admits a semi-orthogonal decomposition $\Phi = \la T_1 ,\cdots, T_m \ra$, then we have several ``fractional'' spherical twists $S_i: T_i \to \Psi$, and the monodromy $\mu = \mu_1 \circ \cdots \circ \mu_m$. 

Kite and Segal \cite{kite-segal} identified the possible SOD factors $Z_{/F}$ for $\Coh(Z_W)$, where $F$ runs through minimal faces of the polytope $\Pi$. They have shown that different SODs have a Jordan-H\"older property, namely the multiplicity $n_{W, F}^B$ of $\Coh(Z_{/F})$ in $\Coh(Z_W)$ is well-defined. They conjectured that $n^A_{W,F} = n^B_{W,F}$ and proved it when rank $k$ of the action torus is $\leq 2$. 

\subsubsection{Statement of Results}
Our first result is about any toric GIT problem, possibly non-CY. 

Let $Q: \Z^N \to \Z^k$ be a toric GIT problem, $C$ a chamber in the GKZ fan and let $H \In \R^k$ be a \emph{relevant subspace} i.e. $H = \cone(\{q_i: q_i \in H \})$ (Definition \ref{d:minface-relsub}). Let $X_C$ be the GIT quotient corresponding to the chamber $C$, and let $Z_H$ be the SOD component for the relevant subspace $H$. We are interested in the SOD multiplicity of $\Coh(Z_H)$ in $\Coh(X_C)$, denoted as $[X_C: Z_H]$.  Let $[N]_H = \{i \in [N]: q_i \in H \}$ be the index set of the weights lying on $H$. Then we have a \emph{Coulomb GIT problem} (Section \ref{ss:Coulomb})
$$ Q_{/H}: \Z^N / \Z^{[N]_H} \to \Z^k / (\Z^k \cap H)$$
where the chamber $C$ descends to a chamber $C/H$, with corresponding quotient $X_{C/H}$, the relevant subspace $H$ quotient to a point $H/H$, and $Z_{H/H}=pt$. 

\begin{theo}[Lemma \ref{l:B-pullback}]
Let $Q$ be any toric GIT problem, $C$ a chamber, $H$ a relevant subspace. Then the multiplicity is invariant under passing to the Coulomb problem $Q_{/H}$
$$[X_C: Z_H] = [X_{C/H}: Z_{H/H}]. $$
\end{theo}
Combining the above result, and the result in Horja-Katzarkov \cite{horja2022discriminants}
\begin{equation}
    \sum_F n^A_{W,F} \rank(Z_{/F}) = \sum_F n^B_{W,F} \rank(Z_{/F}) \label{eq:HK}
\end{equation} 
we can easily get the following main theorem
\bt[Theorem \ref{t: main}]
Let $W$ be a codimension one cone in the GKZ fan of a CY problem. For any minimal face $F$ of the polytope $\Pi$, let $n_{W,F}^A$ denote the intersection multiplicity (A-side multiplicity) defined by the tropical complex, and $n_{W,F}^B$ denote the semiorthogonal decomposition multiplicity (B-side multiplicity). Then $$n_{W,F}^A=n_{W,F}^B.$$
\et

In addition, we obtain a recursive formula relating multiplicities to the ranks of K-theory, hence volumes of stacky fans (see Section \ref{s:volume-of-fan}).

\begin{theo}[Proposition \ref{p:recursive}]
Let $C$ be a GKZ chamber, and $\bSigma$ the corresponding stacky fan. For each minimal face $F \in \fcal$, let $n_F:=[X_\bSigma: Z_{/F}]$ denote the SOD multiplicity of $\Coh(Z_{/F})$ in $\Coh(X_\bSigma)$. Then we have a system of linear equations labelled also by minimal faces $\fcal$, 
\begin{equation}
\rank([X_{\bSigma \cap F}]) := \sum_{F' \leq F} n_{F'} \rank([Z_{F/F'}]). 
\end{equation}
where equation labelled by $F$ only involves $n_{F'}$ with face $F' \leq F$.
\end{theo}

\subsection{Related Work}
Homological mirror symmetry for toric variety has been extensively studied, using Floer theoretic technique \cite{abouzaid2006homogeneous, hanlon2019monodromy, hanlon2020functoriality} and using the microlocal sheaf method \cite{FLTZ-morelli, FLTZ-hms, Ku16, zhou2019twisted, GPS3}, see a recent review using GIT quotient \cite{shende2021toric}. 

In the context of mirror symmetry to toric GIT problem, it is well-understood on the B-side \cite{herbst2008phases, BFK, halpern2015derived, segal2011equivalences} how to use window subcategory to do wall-crossing between adjacent chambers, and sometimes in nice cases (e.g. quasi-symmetric case) how to do wall-crossing among all chambers simultaneously \cite{halpern2020combinatorial, vspenko2019quasi-symm}. 

On the A-side mirror to toric GIT, if we use microlocal sheaf as the A-model, then thanks to the functoriality of coherent constructible correspondence \cite{bondal2006derived, FLTZ-morelli, Ku16, zhou2019twisted}, we can translate all the B-side VGIT result to the A-side \cite{zhou2020-flipflop, huangzhou2020-qsymm}, window categories into window skeletons etc. If we use the more traditional Fukaya category A-model, then the question is more subtle and harder to solve \cite{DKK, BFDKK, kerr2017homological}. The program of Ballard-Diemer-Favero-Katzarkov-Kerr is about matching the A-side and B-side SOD. 

In general, the above result on B-model (and on microlocal A-model) only sees the codimension-1 wall crossing, and did not see the entire GKZ discriminant locus, the recent work of Kite-Segal \cite{kite-segal} in some sense remedies the above deficiency. 

The numerical version of the full conjecture, i.e. $n^A_{W,F} = n^B_{W,F}$,  has been proven by Kite-Segal in the case when $\dim_\C T=1,2$. And Horja-Katzarkov \cite{horja2022discriminants} proved an integrated version of the desired equality (Eq. \ref{eq:HK}),  which does most of the heavy-lifting for us. 

To get the full categorical conjecture \cite{aspinwall2017mirror, kite-segal, horja2022discriminants}, one needs to identify the A-model SOD component mirror to the B-model counterpart $Z_{/F}$, and find a way to book-keep the relations between various SOD components. We leave these to future work. 

\subsection{Outline}
We will mainly work on the B-side. In the next section we introduce the necessary notation for toric GIT and review the notion of Coulomb and Higgs problems for a toric GIT, then we prove our main Lemma. Most of the content are review or slight generalization of \cite{kite-segal}. Then, in the last section, we compare the A-side and B-side multiplicity for toric CY GIT wall-crossing. There are examples at the end of each section, which might help to counter the heavy notations. 

\subsection{Acknowledgements}
We thank the referee for carefully reading our paper and providing valuable comments.

\section{The B-side}
In this section, we first setup the general toric GIT problem, then introduce the GKZ fan to relate various phases of GIT quotients (M-side). These data can be equally well encoded in certain triangulations problem (N-side). Next, we introduce the notion of a Coulomb-Higgs GIT problem, which represent a sub-quotient of the original GIT problem. See \cite{kite-segal, horja2022discriminants, GKZ-book} for more details. Then we prove our main lemma, namely the SOD multiplicity is invariant under passing to the Coulomb problem associated to a minimal face. Given the lemma, we then deduce a recursive formula for SOD multiplicities.

\subsection{Toric GIT setup}
Our starting point is a torus $(\C^*)^k$ acting on $V=\C^N$ with weights $q_1, \cdots, q_N \in \Z^k$. More invariantly, we have a rank $k$ lattice $L \cong \Z^k$, and $L_\CS = L \otimes_\Z \C^* \cong (\C^*)^k$ acts on $\C^N$ by factoring through $L_\CS \to (\C^*)^N$. This induces a map on the cocharacter lattices and character lattices 
%\footnote{We could have denoted the co-character lattice $X_*((\CS)^N) = \Hom(\C^*, (\CS)^N)$ by $\Z^N$, and the character lattice $X^*((\CS)^N)=\Hom((\CS)^N, \CS)$ by $\ZvN$, but to unclutter notation, we drop the dual sign and denote both lattices simply as $\Z^N$.}
$$ Q^\vee: L \to \Z^N, \quad Q: (\Z^N)^\vee \to \Lv. $$
We assume the $\coker(Q)$ is finite, or equivalently $Q^\vee$ is injective. 

We have a short exact sequence (SES), called 'N-sequence'
$$ 0 \to L \xto{Q^\vee} \Z^N \xto{A} \NN \to 0 $$
where $\NN$ might have torsion 
$$ 0 \to \NN_{tors} \to \NN \to \NN_{free} \to 0. $$
Let $\bar A: \Z^N \to \NN_{free}$ denote the obvious composition. 
Apply the $\Hom(-, \Z)$ to the 'N-sequence',  and use injective resolution $\Z \to \Q \to \Q / \Z$, then we have a long exact sequece,  the 'M-sequence'
$$ 0 \to \MM_0 \xto{A^\vee} (\Z^N)^\vee \xto{Q} \Lv \to \MM_1 \to 0, $$
where $\MM_0 = \Hom(\NN, \Z)=\Hom(\NN_{free}, \Z)$ is a lattice, and $\MM_1=\Ext^1(\NN, \Z) = \Hom(\NN_{tors}, \Q/\Z)$ is the Pontryagin dual to the torsion subgroup $\NN_{tors}$. 

If $G$ is an abelian group, we denote $( \cdots )_G:= ( \cdots ) \otimes_\Z G$, where $G$ can be $\R, \Z, \T = \R/\Z, \C, \C^*$. For example, if we apply $(\cdots)_\R$ to $Q: (\Z^N)^\vee \to \Lv$, then we will get $Q_\R: (\R^N)^\vee \to \Lv_\R$. 

Let the co-character lattice $X_*((\CS)^N) \cong \Z^N$ be equipped with the standard basis $e_1, \cdots, e_N$, and let the character lattice $X^*( (\CS)^N) \cong (\Z^N)^\vee$ be equipped with the dual basis $e_1^\vee, \cdots, e_N^\vee$. Define
$$ \wt a_i = A(e_i) \in \NN, \quad a_i = \bar A(e_i) \in \NN_{free},  \quad q_i = Q(e_i^\vee) \in \Lv. $$

% {\red do we need these combinatorics? Maybe not. }
% Let $[N] = \{1, \cdots, N\}$ and let $\ical := 2^{[N]} := \{I \mid I \In [N]\}$ denote the collection of subsets in $[N]$.   For any $0 \neq I \In [N]$, we define relative open cones
% $$ \sigma_I = \sum_{i \in I} \R_{>0} \cdot e_i \In \R^N, \quad  \tau_I = \sum_{i \in I} \R_{>0} \cdot e_i^\vee \In (\R^N)^\vee, $$
% and we let $\sigma_\emptyset = 0, \tau_\emptyset = 0$. 

% {\red using open cones are weird. maybe we should use strata? }
% By default, all cones and faces in this paper are relatively open cones, where a subset $S \In \R^N$ is relatively open if $S$ is open in the affine linear space generated by $S$. And we denote their closure with an overbar. 

\subsection{M-side, chambers and walls}
By the M-side, we mean the objects living on spaces in the M-sequence, for example about the map  $Q_\R: (\R^N)^\vee \to \Lv_\R$. Here for the simplicity of notation, we write $(\R^N)^\vee$ as $\R^N$. 

We recall the definition of GKZ fan (rather GKZ stratification) $\Sigma_{GKZ}(Q)$ in $\Lv_\R$ (aka GIT fan or secondary fan) for a toric GIT problem $Q$. We omit $(Q)$ if there is no danger of confusion.

Let $\mathscr P(S)$ denote the power set of a set $S$. By the GKZ stratification $\Sigma_{GKZ}(Q)$, we mean the stratification of $\Lv_\R$ given by the level sets of the following map:
\begin{eqnarray*}
L^\vee_\mathbb R & \to & \mathscr P (\{ \tau | \tau \text{ is a closed face of } (\mathbb R^N_{\geq 0})^\vee \} )\\
x & \mapsto & \{\tau | x\in Q_\mathbb R (\tau ) \}.
\end{eqnarray*}

% the following equivalence relation among points in $L^\vee_\mathbb R$:

% $$
% x\simeq x' \Leftrightarrow \{\tau \text{ closed face of } (\mathbb R^N_{\geq 0})^\vee | x\in Q_\mathbb R (\tau ) \}=\{\tau \text{ closed face of } (\mathbb R^N_{\geq 0})^\vee | x'\in Q_\mathbb R (\tau ) \}.
% $$
% generated by $Q_\R$ images of various faces of the positive quadrant $P_N = \R_{\geq 0}^N$. 
A top dimensional stratum is called a chamber, and a codimension-1 stratum is called a wall. 
The support of the GKZ fan is $Q_\R(\R_{\geq 0}^N)$. If the support is not the full $\Lv_\R$, then its complement is still a GKZ strata, and we call it the empty chamber.

Here, the positive quadrant $P_N=(\R_{\geq 0}^N)^\vee$ can be identified as the image of the moment map of $(\C^*)^N$ acting on $\C^N$
$$ \mu_N: \C^N \to \R^N, \quad (z_i)_i \mapsto (|z_i|^2)_i. $$
And the moment map of $L_\CS$ acting on $\C^N$ is $\mu = Q_\R \circ \mu_N$. 

For any $c \in \Lv_\R$, we define the GIT quotient stack
$$ X_c = [(\C^N)^{ss}_c / L_\CS],  \quad (\C^N)^{ss}_c = \{z \in \C^N \mid \overline{L_\CS \cdot z} \cap \mu^{-1}(c) \neq \emptyset \}.$$
More concretely, the positive quadrant $P_N$ is stratified by faces $\tau$, $P_N = \sqcup \tau$  which induces a stratification of $\C^N$ into strata $(\C^N)_\tau = \mu_N^{-1}(\tau)$, and $(\C^N)^{ss}_c = \sqcup \{(\C^N)_\tau \mid Q(\overline \tau) \cap c \neq \emptyset\}.$ From the latter description, it is clear that $X_c$ is constant when $c$ varies within a GKZ stratum $C$, hence we also write $X_c$ as $X_C$. For any GKZ stratum $C$, we call the GIT quotient $X_C$ a ``phase'' of the toric GIT problem. 

Let $\det V= \sum_i q_i \in \Lv$ denote the weight of $L_\CS$ acting on $\det V$. The toric GIT is called Calabi-Yau (CY), if $\det V=0$. 

Let $W$ be a wall separating two chambers $C_+, C_-$. We choose the $\pm$ sign so that on the wall $\det V$ is pointing towards $C_+$. Let $\lambda_W \in L$ be an integral primitive vector conormal to the wall. Let $d_W = |\la \lambda_W, \det V \ra |$.
If $d_W = 0$, i.e, $\det V$ is parallel to $W$,  we say there is a balanced wall-crossing and we have (non-canonical) derived equivalence $\Coh(X_{C_+}) \cong \Coh(X_{C_-})$. If $d_W > 0$, then we have semi-orthogonal decomposition
$$ \Coh(X_{C_+}) = \la \Coh(X_{C_-}), \Coh(Z_W), \cdots , \Coh(Z_W) \ra, $$
where $Z_W$ is defined in \eqref{eq:ZW}, and the factor $\Coh(Z_W)$ repeats $d_W$ times. 

%{\red We will abuse notation and let $X$ also denote $\Coh(X)$.}

Therefore, starting from any point $c \in \Lv_\R$, we may form the ray in the direction of $-\det V$. This is called the ``straight-line'' run in \cite{DKK}. If $c$ is generic, then the run will only encounter walls, and ends in a chamber $C$ that contains $-\det V$ in its closure. We call such a chamber $C$ and the corresponding phase $X_C$ minimal. It is possible that $X_C = \emptyset$ or that there are several minimal chambers, but upto derived equivalence, the minimal phase is unique. Let $X_{\min}(Q)$ denote the minimal phase (possibly empty) for toric GIT problem. 

\begin{remark}
Here are some side remarks on the relations between CY and non-CY toric GIT problems, which will not be used in the rest of the paper. 

For a non-CY problem, we can associate a CY problem to by adding a weight $q_{N+1} = -\det V$. The new problem is
$$ 0 \to L \xto{\hat Q} \Z^{N+1} \xto{\hat A} \NN \oplus \Z \to 0 $$
We have $\hat A(e_i) = (A(e_i), 1)$ for $i \in [N]$, and $\hat A(e_{N+1}) = (0, 1)$. 

Going the other way around, suppose we have a toric CY GIT problem, then pick any $q_i$ (which by assumption are all nonzero), say $q_{N+1}$ by relabelling,  we can delete it and get a toric non-CY GIT problem. 

Assuming $Q$ is a non-CY problem, then the toric CY GKZ stratification $\Sigma_{GKZ}(\hat Q)$ refines the non-CY GKZ stratification $\Sigma_{GKZ}(Q)$. The new walls in 
$\Sigma_{GKZ}(\hat Q)$ are all parallel to $\det V$. 
\end{remark}

% Later we will introduce a Higgs problem associated to the hyperplane generated by the wall, and define $Z_W$ as the phase of the toric variety.

% If tangent vector $-\det V$ is transverse to the wall $W$, we choose $C_\pm$ so that $-\det V$ points from $C_+$ to $C_-$.

%Cones in $\Sigma_{GKZ}$ are intersections of collections $Q_\R(\tau_I)$, where $\tau_I$ are faces of the closed positive quadrant $P=\R_{\geq 0}^N$. 

% Let $C$ be a chamber, we define a sub-collection $\ical_C$ of faces of $P$, 
% $$ \ical_C = \{ I \in \ical \mid \bar \tau_I \cap Q_\R^{-1}(C) \neq \emptyset \}. $$
% Consider the moment map for $(\C^*)^N$ acting on $\C^N$, $\mu: \C^N \to P$, $(z_i)_i \mapsto (|z_i|^2)_i$, then for each face $\tau_I$ of $P$, we define
% strata $(\C^N)_I = \mu^{-1}(\tau_I)$. Then the semi-stable loci of $\C^N$ for chamber $C$
% is 
% $$ (\C^N)^{ss}_C = \bigsqcup_{I \in \ical_C} (\C^N)_I.$$
% We may define the GIT quotient stack for this chamber as
% $$ X_C := [(\C^N)^{ss}_C / L_\CS]. $$
% $X_C$ is also called a {\bf phase} of the GIT problem. 
% The above definition also works for $C$ being any GKZ strata. 

%A chamber $C$ is a minimal chamber if it contains $-\det V$ in its closure; correspondingly a chamber $C$ is maximal if it contains $\det V$ in its closure. The minimal and maximal phases are well-defined up to derived equivalence. 

%If $\det V=0$, we say the toric GIT problem is Calabi-Yau (CY). 

\subsection{N-side, local triangulations and stacky fan}
By the N-side, we mean the objects living on spaces in the N-sequence, for example  the map  $A: \Z^N \to \NN$. Recall $a_i = \bar A(e_i) \in \NN_{free} \In \NN_\R$ and $\wt a_i = A(e_i) \in \NN$. 

\subsubsection{Localized Marked Polytope Subdivision}
Let $\Delta_N = \conv(0, e_1, \cdots, e_N)$ be the standard $N$-simplex in $\R^N$. Let $\Pi = A_\R(\Delta_N)$ be the image of $\Delta_N$ in $\NN_\R$, which is also the convex hull $\conv(0, a_1, \cdots, a_N)$. 

We start with a pair of piecewise linear (PL) functions on $\R^N$ and $(\R^N)^\vee$ related by Legendre transformation. Let 
$$ \varphi_N(x_1, \cdots, x_N) := \min(0, x_1, \cdots, x_N) :  (\R^N)^\vee \to \R. $$ 
It is a concave function, with $\varphi_N^{-1}(0) = P_N = (\R_{\geq 0})^N$. By Legendre transformation, we have
$$ \psi_N(y_1,\cdots,y_N) = \max_{x \in \R^N}(\varphi(x) - (x, y) ): \R^N \to \R. $$ 
$\psi_N$ is a convex function, and explicitly
$$ \psi_N(y)=\begin{cases}
0 & y \in \Delta_N \\
+\infty & y \notin \Delta_N
\end{cases} $$

Given any point $c \in \Lv_\R$, choose a lift $ b \in Q^{-1}(c)$. Then, we can identify $\MM_\R \cong Q^{-1}(c)$ by 
$$ f_{b}: \MM_\R \to Q^{-1}(c), \quad  \xi \mapsto b + A_\R^\vee(\xi) = (b_1 + (a_1, \xi), \cdots, b_N + (a_N, \xi)). $$  The restriction of $\varphi_N$ on $Q^{-1}(c)$ pullback to $\MM_\R$ by $f_{b}$ gives 
$$ \varphi_{b}(\xi) := \varphi_N(b + A_\R^\vee(\xi)) = \min(0, b_1 + a_1 (\xi), b_2 + a_2 (\xi), \cdots, b_N + a_N ( \xi)). $$
Its Legendre transformation is
$$ \psi_{b}(\eta) = \max_{\xi \in \MM_\R}(\varphi_{b}(\xi )- (\eta , \xi) ). $$

{\bf Notation:} For uniform treatment later, it is useful to introduce another pair $(a_{N+1}, b_{N+1}) = 0 \in \NN_\R \times \R$ for the origin. 

The overgraph of $\psi_b$
$$ \Gamma_{\geq  \psi_{b}} := \{(x,y) \in  \Pi \times \R \mid y \geq \psi_{b}(x) \} $$ 
is the convex hull of the upward rays 
$$ \bigcup_{i=1}^{N+1}  a_i \times \{ y \geq  b_i\}  \In \NN_\R \times \R.$$ 
We record the subset $S \In [N+1]$ where the tip of the ray touches the graph, i.e. 
$$ S = \{ i \in [N+1] \mid \psi_b(a_i) = b_i \}. $$ 
For $x \in \Pi$, if $x = a_i$ for some $i \in S$, we say {\bf $x$ is marked by $i$}. It is possible $x$ is marked by more than one $i$. 

The maximal linearity domain of $\psi_b$ defines a polytope subdivision of $\Pi$, and the vertices of the polytopes and possibly some interior points are marked by some $i \in S$. 

The polytope subdivision together with marking points $S$ is called a coherent (multi-)marked polytope subdivision, and marked subdivision for short. 

If we choose a different lift $b \in Q^{-1}(c)$, then $\psi_b$ only change by a linear function and does not affect the coherent subdivision. For a generic choice of $c$, we have triangulations of $\Pi$ with no interior marked points, and each vertex of the simplices is marked exactly once. 

%We say a marked subdivision uses $a_i$ if $a_i$ (rather $i$) appears in the marking. 

\begin{definition}
Let $x$ be any point in $\Pi$, we say two marked subdivisions are {\bf equivalent at $x$} if they have the same collection of marked polytopes that contain $x$. 

A {\bf  marked subdivision localized at $x$} is an equivalence class of marked subdivisions modulo equivalence relations at $x$.  
\end{definition}

For any $c \in \Lv_\R$, we may use the above procedure to obtain a marked subdivision $T(c)$ of $\Pi$. 

%{\red Explain $\Sigma_{GKZ}$ to referee then the proof should be fine}

\begin{proposition}
\begin{enumerate}
    \item For any $c \in \Lv_\R$, the subdivision $T(c)$ has a marked point at the origin if and only if $c$ is not in the empty chamber of $\Sigma_{GKZ}(Q)$. 
    \item For any $c_1, c_2 \in \Lv_\R$,  $c_1$ and $c_2$ are in the same GKZ stratum of $\Sigma_{GKZ}(Q)$ if and only if $T(c_1)$ is equivalent to $T(c_2)$ at the origin. 
\end{enumerate}
\end{proposition}
\begin{proof}
 For (1), $T(c)$ has a marked point at the origin if and only if $\psi_b(0)=0$, which is equivalent to $\max_{\xi \in M_\R} \varphi_b(\xi) = 0$, and is equivalent to $Q_\R^{-1}(c) \cap \R_{\geq 0}^N \neq \emptyset$. 
 
 For (2), $c_1$ and $c_2$ are in the same GKZ stratum, if and only if the fibers $Q_\R^{-1}(c_1)$ and $Q_\R^{-1}(c_2)$ intersect each face $\tau$ of $\R_{\geq 0}^N$ in the same way. Each non-empty intersection of $\tau \cap Q_\R^{-1}(c)$ corresponds to a polytope (possibly not of top dimension) in $T(c)$ containing $0$. 
\end{proof}

\begin{remark}
If the GIT problem $Q$ is Calabi-Yau, then any marked subdivision of $\Pi$ uses the origin and the equivalence relation is trivial. In this case, we have a bijection between GKZ chambers and marked triangulation of $\Pi$. 

If the GIT problem $Q$ is not CY, then we may consider the associated CY problem $\hat Q$. Then chambers in $\Sigma_{GKZ}(\hat Q)$ corresponds to marked triangulation of of $\Pi$; and chambers in $\Sigma_{GKZ}(Q)$ corresponds to marked triangulations of $\Pi$ localized at $0$. 
\end{remark}

\begin{remark}
Since the cones in the toric fan are precisely the cones on the (marked) polytopes that contain the origin, polytopes that do not contain the origin do not affect the toric data for that phase.
\end{remark}

\begin{remark}
    We remind the reader that, in comparison with \cite{kite-segal} in the CY case, our polytope $\Pi$ is not the same $\Pi$ that appears in \cite{kite-segal}, but is the cone on it from the origin. Furthermore, since all faces in our marked subdivision include the origin, these faces of our $\Pi$ are in bijection with the faces of $\Pi$ as in \cite{kite-segal}.
\end{remark}

\subsubsection{Stacky Fan and its Volume}\label{s:volume-of-fan}
Let $C$ be a chamber of GKZ fan, and let $T(C)$ be the corresponding marked triangulation localized at the base point. This defines a stacky fan in $\NN$ as follows. 

Recall the definition of simplicial stacky fan following \cite{borisov2005orbifold}. Let $\NN$ be a finitely generated abelian group, $\Sigma$ be a rational simplicial polyhedran fan in $\NN_\R$, and $\{v_i\}_{i \in \Sigma^1} \In \NN$ such that $(v_i)_\R$ generate the corresponding ray in $\Sigma$. The triple $\bSigma = (\NN, \Sigma, \{v_i\})$ is called a stacky fan. 

Roughly speaking, modulo torsion, we can think of a stacky fan $\bSigma$ as a collection of simplices $\bsigma \In \NN_\R$ with shared vertices $(v_i)_\R$. 

We write $|\Sigma|$ for the union of the cones $\sigma$, and we write $|\bSigma|$ for the union of the simplices $\bsigma$. 

We normalize volume $\vol_{\R}$ on $\NN_\R$ such that a minimal simplex with vertices in $\NN_{free}$ has unit volume. Then we define the stacky volume by
$$ \vol(\mathbf \Sigma) := |\NN_{tors}| \cdot \vol_\R (|\bSigma|). $$

By design of the volume, we have the following result
\begin{proposition} [\cite{horja2022discriminants}]
The volume of the stacky fan equals the rank of the $K_0$ group of 
$ X_\bSigma$, 
$$\rank(K_0(X_\bSigma)) = \vol(\bSigma). $$
\end{proposition}

\subsubsection{Minimal phases and Minimal Fans} \label{s:volume-of-ZF}
Given a toric GIT problem $A: \Z^N \onto \NN$ or $Q: (\Z^N)^\vee \to \Lv$, on the $M$-side, we have some minimal chambers $C$, i.e. those containing $-\det(V)$. On the N-side, we have some minimal stacky fans corresponding to the minimal chambers. They have the same support $|\bSigma^{min}|$ which we describe now.

Let $S = A_\R(\Z_{\geq 0}^N)$ and $S_+ = A_\R(\Z_{\geq 0}^N \RM \{0\})$ and define (slightly abusing notation)
$$ |\bSigma^{min}| = \conv(S) - \conv(S_+). $$
We see $|\bSigma^{min}| = \emptyset$ if and only if $S = S_+$, or $0 \in \Pi=\conv(a_i)$, or the support of the GKZ fan is not the whole $\Lv_\R$. 

\begin{proposition} [\cite{horja2022discriminants}]
We have the rank-volume relation
$$ \rank(K_0(X_{min})) = \vol(|\bSigma^{min}|) = |\NN_{tors}| \cdot \vol_\R(\conv(S) - \conv(S_+)).$$
\end{proposition}

\subsection{Coulomb and Higgs GIT problems \label{ss:Coulomb}}

We follow \cite{kite-segal, aspinwall2017mirror} and introduce two GIT problems called Coulomb and Higgs problems respectively. 

\subsubsection{Definition using subsets of $[N]$}
Fix any subset $\Ga \In [N]$. We may consider the subspace $\C^\Ga \In \C^N$, and consider the subtorus in $L\otimes \C^*$ that preserves  $\C^\Ga$. The following commutative diagrams might be useful for book-keeping. 

\[
\begin{tikzcd}
& 0 \ar[d] & 0 \ar[d] & 0 \ar[d] & \\
0 \ar[r]  & L_\Ga \ar[r, "Q_\Ga^\vee"]\ar[d] & \Z^\Ga \ar[r, "A_\Ga"]\ar[d] &  \NN_\Ga \ar[r]\ar[d] & 0 \\
0 \ar[r]  & L \ar[r, "Q^\vee"]\ar[d] & \Z^N \ar[r, "A"]\ar[d] &  \NN \ar[r]\ar[d] & 0 \\
0 \ar[r]  & L/L_\Ga \ar[r, "Q_{/\Ga}^\vee"]\ar[d] & \Z^N  / \Z^\Ga \ar[r, "A_{/\Ga}"]\ar[d] &  \NN/\NN_\Ga \ar[r]\ar[d] & 0 \\
  & 0    & 0   & 0 &
\end{tikzcd}
\]

Dualize the first two columns, and then apply the snake lemma, we get
\[
\begin{tikzcd}
& 0 \ar[d] & 0 \ar[d] & 0 \ar[d] & ... \ar[d] &  \\
0 \ar[r]  &  (\NN/\NN_\Ga)^\vee_0 \ar[r, "A_{/\Ga}"]\ar[d] & (\Z^N  / \Z^\Ga)^\vee \ar[r,"Q_\qGa"]\ar[d] &  ( L/L_\Ga )^\vee\ar[r]\ar[d] & (\NN/\NN_\Ga)^\vee_1 \ar[r]  \ar[d] & 0 \\
0 \ar[r]  &  \MM_0 \ar[r, "A^\vee"]\ar[d] & (\Z^N)^\vee \ar[r,"Q"]\ar[d] &  \Lv \ar[r]\ar[d] & \MM_1 \ar[r]  \ar[d] & 0 \\
0 \ar[r]  &  (\MM_\Ga)_0 \ar[r,"A_\Ga^\vee"]\ar[d] & (\Z^\Ga)^\vee \ar[r,"Q_\Ga"]\ar[d] &  \Lv_\Ga \ar[r]\ar[d] & (\MM_\Ga)_1 \ar[r] \ar[d] & 0 \\
  & ...    & 0   & 0 & 0 & 
\end{tikzcd}
\]

\begin{definition} \cite[Section 2.3]{kite-segal}\label{def: Coulomb-Higgs}
Let $\Ga \In [N]$. We call the GIT problem associated to $Q_\Ga$ (or $A_\Ga$)
as the {\bf Coulomb problem} for $\Ga$, and the GIT problem associated to $Q_{/\Ga}$ and $A_{/\Ga}$ as the {\bf Higgs problem} for $\Ga$. 

More generally, if we have $\Ga_1 \In \Ga_2 \In [N]$, we may form the GIT problem 
$$   L_{\Ga_2} / L_{\Ga_1} \xto{Q^\vee_{\Ga_2 / \Ga_1}} \Z^{\Ga_2} / \Z^{\Ga_1} \xto{A_{\Ga_2/\Ga_1}} \NN_{\Ga_2} / \NN_{\Ga_1} $$ as the Coulomb-Higgs problem for the pair $\Ga_1 \In \Ga_2$.
\end{definition}

The Coulomb problem for $\Ga$ is the subtorus $L_\Ga \otimes \C^*$ acting on the subspace $\C^\Gamma$. The Higgs problem for $\Ga$ is the quotient torus $L_{\C^*} / (L_\Ga)_{\C^*}$ acting on the subspace $\C^{\Gamma^c}$ that is fixed by $(L_\Ga)_{\C^*}$.

\subsubsection{Definition using subspaces and faces}
Although the Coulomb-Higgs problems can be defined for general subsets $\Ga_1 \In \Ga_2 \In [N]$, they often arise from subspaces $H \In \Lv_\R$ and faces $F \In \Pi$.

\begin{definition}
A subspace $H \In \Lv_\R$ is called {\bf weight supported} if $H = \span_\R\{q_i \in H \}$. 

A subset $F \In \Pi$ is called a {\bf face} if there is a linear function $l: \NN_\R \to \R$, such that $l|_\Pi \geq 0$ and $F=l^{-1}(0) \cap \Pi$.
\end{definition}

Recall $[N]_H = \{i: q_i \in H\}$ and $[N]_F = \{i: a_i \in F\}$.

\begin{definition}
If $H_1 \In H_2 \In \Lv_\R$ are a pair of weight supported subspaces, we define the Coulomb-Higgs problem $Q_{H_2/H_1}$ 
$$ Q_{H_2/H_1}: \Z^{[N]_{H_2} - [N]_{H_1}} \to (\Lv \cap H_2)/H_1. $$
If $H_1 =0$, we get a Higgs problem $Q_{H_2}$; if $H_2=\Lv_\R$, we get a Coulomb problem $Q_{/H_1}$. 

%Here in the M-side, taking subspace $H_2$ corresponds to passing to Higgs problem for subset $[N]_{H_2}^c$, and taking quotient $/H_1$ corresponds to passing to Coulomb problem for subset $[N]_{H_1}^c$. 

If $F_1 \In F_2 \In \Pi$ is a pair of faces, we define the Coulomb-Higgs problem $A_{F_2/F_1}$
$$ A_{F_2/F_1}: \Z^{[N]_{F_2} - [N]_{F_1}} \to \NN_{F_2} / \NN_{F_1}, \quad \NN_{F} := \NN_{\Ga = [N]_F}.  $$ 
If $F_1 = 0$, we get a Coulomb problem $A_{F_2}$; if $F_2=\Pi$, we get a Higgs problem $A_{/F_1}$. 
\end{definition}

%Here in the M-side, taking face $F_2$ corresponds to passing to Coulomb problem for subset $[N]_{F_2}$, and taking quotient $/F_1$ corresponds to passing to Higgs problem for subset $[N]_{F_1}$. 

%Consider the M-side. Let $H$ be a subspace spanned by $q_i$ contained in it (we call such subspace {\bf weight supported}),  then we call the Coulomb/Higgs problem associated to $\Ga = [N]_H^c$ as the Coulomb/Higgs problem associated to $H$. Concretely, the Higgs GIT problem is $Q_H : \Z^{[N]_H} \to \Lv \cap H$, and the Coulomb GIT problem is $Q_{/H}: \Z^{[N]_H^c} \to \Lv / H$. 

%On the M-side, the GKZ stratification $\Sigma_{GKZ}(Q_H)$ is a coarsening of the intersection $\Sigma_{GKZ}(Q) \cap H$.

%If $H_1 \subset H_2$ are two weight supported subspaces, then we may define the Coulomb-Higgs problem

\subsubsection{Minimal faces and relevant subspaces}
Here we again follow \cite{kite-segal, aspinwall2017mirror} and introduce certain important Coulomb Higgs problem, that describes the SOD components. 

\begin{definition}
Let $S$ be a subset of $[N]$. Recall that $a_i = A_\R(e_i)$, $q_i = Q_\R(e_i^\vee)$. 
\begin{enumerate}
    \item We say $S$ is {\bf $A_\R$-redundant}, if there exists $c_i \neq 0$ for each $i \in S$, such that $\sum_{i \in S} c_i a_i = 0$. 
    \item We say $S$ is {\bf $A_\R$-saturated}, if there is a linear function $l: \NN_\R \to \R$, such that $S = \{i: l(a_i)=0\}$. 
    \item We say $S$ is {\bf $A_\R$-extremally-saturated}, if there is a linear function $l: \NN_\R \to \R$, such that $S = \{i: l(a_i)=0\}$, and $l(a_i) > 0$ for all $i \in S^c$. 
    \item We say $S$ is {\bf $Q_\R$-redundant}, if there exists $c_i \neq 0$ for each $i \in S$, such that $\sum_{i \in S} c_i q_i = 0$. 
    \item We say $S$ is {\bf $Q_\R$-saturated}, if there is a linear function $l: \Lv_\R \to \R$, such that $S = \{i: l(q_i)=0\}$. 
    \item We say $S$ is {\bf $Q_\R$-positively-redundant} if there exists $c_i > 0$ for $i \in S$, such that $\sum_{i \in S} c_i q_i = 0 $. 
\end{enumerate} 
\end{definition}

\begin{proposition}
Let $S$ be a subset of $[N]$.
\begin{enumerate}
    \item $S$ is $A_\R$-redundant if and only if $S^c$ is $Q_\R$-saturated. 
    \item $S$ is $Q_\R$-redundant if and only if $S^c$ is $A_\R$-saturated. 
    \item $S$ is $Q_\R$-positively redundant if and only if $S^c$ is $A_\R$-extremally-saturated. 
\end{enumerate}
\end{proposition}
\begin{proof}
% We only prove the first statement, as the second follows verbatim by exchanging $Q$ and $A$. 

For (1), by definition, $S$ is $A_\R$-redundant, if there exists $c_i \neq 0$ for each $i \in S$, such that $\sum_{i \in S} c_i a_i = 0$. By setting $c_i=0$ for $i \notin S$, we get an element $\vec c = (c_i) \in \R^N$, such that $A(\vec c)=0$, i.e, $\vec c \in L$. Hence $\vec c$ defines a linear function $l: L^\vee_\R \to \R$, such that $l(q_i)=0$ if and only if $c_i=0$, i.e $i \in S^c$. Thus $S^c$ is $Q_\R$-saturated. The argument can also be reversed, hence we get the equivalence. 

For (2), we only need to change the above argument by swapping $A_\R$ with $Q_\R$, and  $a_i$ with $q_i$. 

For (3), we can check that the two positivity conditions match. 
\end{proof}

Recall $\Pi = A_\R(\Delta^N) \In \NN_\R$. A subset $F \In \Pi$ is called a face is there is a linear function $l: \NN_\R \to \R$, such that $l(\Pi) \geq 0$ and $F=l^{-1}(0)$. For example, $\Pi$ is always a face of itself. And if $\Pi$ contains a linear subspace $H$, then any face $F \supset H$.

\begin{definition} \label{d:minface-relsub}
Let $F$ is a face of $\Pi$ and define $[N]_F = \{ i: a_i \in F\}$. By construction $[N]_F$ is $A_\R$-extremally-saturated. 
We say $F$ is a {\bf minimal face} of $\Pi$ if $[N]_F$ is $A_\R$-redundant. 

Let $H \In \Lv_\R$ be a weight supported subspace, and define $[N]_H = \{i: q_i \in H\}$. We say $H$ is a {\bf relevant subspace} of $\Lv_\R$ if $[N]_H$ is $Q_\R$-positively-redundant and $Q_\R$-saturated. 
\end{definition}

\begin{remark}
Our definitions generalize the definitions of minimal faces and relevant subspaces in \cite{kite-segal} in the toric CY setting.
\end{remark}

We obtain the following slight generalization of Proposition 4.15 in \cite{kite-segal}. 
\begin{proposition}
There is a bijection between minimal faces of $\Pi$ and relevant subspaces of $\Lv_\R$, such that if a minimal face $F$ corresponds to a relevant subspace $H$ then $[N]_F = [N]_H^c$.
\end{proposition} 
\begin{proof}
Let $F$ be a minimal face, then $[N]_F$ is $A_\R$ extremally-saturated and redundant, hence $[N]_F^c$ is $Q_\R$-positively-redundant and saturated. Let $H = \span \{q_i: i \in [N]_F^c\}$, we then get a relevant subspace. Going backward is similar. 
\end{proof}

If $F$ is a minimal face (of $\Pi$), we will abuse notation and also call $[N]_F$ a minimal face.  Similarly, if $H$ is a relevant subspace, we will also call $[N]_H$ is a relevant subspace.

\subsubsection{GIT quotients}
Here we define some GIT quotients associated with the Higgs and Coulomb problems.

\begin{definition} \label{d:Z-variety}
If $H$ is a relevant subspace, and $F$ is the corresponding minimal face, with $\Ga=[N]_F$, then we use $Z_H = Z_{/F} = Z_{/\Ga}$ to denote the minimal phase in the Higgs problem $Q_{H}$ and $A_{/F}$. 

In general, for any Coulomb-Higgs GIT problem $Q_{H_1/H_2}$ (resp.  $A_{F_1/F_2}$), we use $Z_{H_1/H_2}$ (resp. $Z_{F_1 / F_2}$) to denote the minimal phase in that problem. 

If $W$ is a wall in $\Sigma_{GKZ}(Q)$, let $H = \span_\R(W)$, then we define $Z_W$ as the phase $W$ of the Higgs problem $Q_H$. 
\end{definition}

The following result and proof is essentially also due to \cite{kite-segal}. 
\begin{proposition}
Let $F$ be a face of $\Pi$, and $A_F$, $Q_F$ be the Coulomb problem associated to $F$. Then we have map of lattices
$$ \pi_F: \Lv \to \Lv_{F}. $$
The map $(\pi_F)_\R$ is compatible with the GKZ stratifications $\Sigma_{GKZ}(Q)$ and $\Sigma_{GKZ}(Q_F)$, i.e. image of a strata is a strata. 
\end{proposition}

\begin{proof}
Let $\wt \Sigma_{GKZ}(Q)$ be the pullback stratification of $\R^N$, and $\wt \Sigma_{GKZ}(Q_F)$ be that of $\R^{[N]_F}$. Suffice to show that under the quotient map $\wt \pi_F: \R^N \to \R^{[N]_F}$, the image of a strata is a strata. 

Let $b \in \R^N$, and $b_F = \pi_F(b) \in \R^{[N]_F}$. Then $b$ defines a PL convex function $\psi_b$ on $\Pi$, and $b_F$ defines a PL convex function $\psi_{b_F}$ on $F$. It is easy to check $\psi_{b_F} = \psi_b|_F$. Hence when $b$ varies within a strata, ie, induced localized marked subdivision remains invariant, the boundary $F$ subdivision also remains invariant. Thus $\pi_F$ sends a strata inside a strata. 

On the other hand, consider the lift (a section of $\pi_F$)
$$ \iota_F: \R^{[N]_F} \to \R^{N} , \quad x \mapsto (x-b_F)+b. $$
We claim that $\iota_F$ also sends a strata into a strata. Indeed, 
as we vary $b_F$ such that the localized marked subdivision of $F$ remains invariant, then we can extend the variation of $b_F$ to that of $b$ by keeping other $b_i$s not in the face $F$ fixed. The resulting localized subdivision of $\Pi$ remains fixed. 
\end{proof}

Given the above result, if $C$ is a chamber corresponding to the and $F$ is a face, then we have the a well-defined chamber $C_F$ in the Coulomb problem $Q_F$. If $C$ corresponds to a stacky fan $\bSigma$, then we use $\bSigma_F$ for $C_F$. Concretely,  $\bSigma_F$ is the stacky fan $(\NN_\Ga, \Sigma \cap (\NN_\Ga)_\R, S \cap \Ga)$.  We restrict the simplices $\bsigma$ to the face $F$ and coarsen $\NN \cap \R F$ to sub-lattice $\NN_\Ga$. %We denote $Z_{F/F}$ to denote the Coulomb-Higgs problem of $A_{F/F}$. 

% \begin{remark}
% The above result would fail if we do not insist on a Coulomb problem associated to a face $F$ but to an arbitrary $\Ga$. For example, consider $q_1=(1,1), q_2(-1,1), q_3=(0,-1), q_4=(0,-1)$, and take the subspace $H = \span_\R q_3$. Then $[N]_H=\{3,4\}$ is $Q$-saturated, redundant (see the next subsection for the definition.), but not positively redundant, meaning $\Ga=[N]_H^c=\{1,2\}$ is $A$-saturated, redundant, but not extremally saturated, i.e $\Ga$ is not a face. The map $\pi_\Ga: \R^2 \to \R, (x,y) \mapsto x$.   We then see the GKZ strata $\cone(q_1, q_2)$ is sent to the entire $\R$, which is a union of 3 stratas $(-\infty, 0), \{0\}, (0,\infty)$. 
% \end{remark}

%Consider the N-side, if we have a face $F$ of $\Pi$, then we may consider the Coulomb problem $A_F: \Z^{[N]_F} \to \NN_F$ corresponding to $\Ga=[N]_F$, and the Higgs problem $A_{/F}: \Z^{[N]_F^c \to \NN / \NN_F}$. 

%If we have two faces $F_1 \In F_2$ of $\Pi$, we can form the Coulomb-Higgs problem $A_{F_2/F_1}: \Z^{[N]_{F_2} - [N]_{F_1}} \to \NN_{F_2} / \NN_{F_1}$. Let $Z_{/F}$ denote the minimal phase in the Higgs problem $A_{/F}$. Similarly for $Z_{F_2/F_1}$.

\subsection{Main Lemma}
First we recall some results from \cite{kite-segal}. 

\begin{proposition}[\cite{kite-segal}]
The SOD components in an ultimate SOD in a toric GIT problem $Q$ are labelled by the set of relevant subspaces, or equivalently, the set of minimal faces. Let $X_C$ be any phase and $H \In \Lv_\R$ any relevant subspaces, then the multiplicity $n^B_{C,H} = [X_C: Z_H]$, as the number of times that $\Coh(Z_H)$ appears in a complete SOD of $\Coh(X_C)$ arising from toric GIT wall-crossing, is well-defined.  
\end{proposition}

Our main lemma equate the multiplicity $[X_C: Z_H]$ to its counterpart $[X_{C/H}: Z_{H/H}]$ in the Coulomb problem $Q_{/H}$. 

Recall we have a quotient map $\pi_H: \Lv_\R \to \Lv_\R/H$. 
As Kite-Segal observed, the quotient $\pi_H$ is compatible with fans $\Sigma_{GKZ}(Q)$ and $\Sigma_{GKZ}(Q_{/H})$. Indeed, a strata in $\Sigma_{GKZ}(Q)$ represent a localized marked subdivision, and passing to $\Sigma_{GKZ}(Q_{/H})$ means we intersect the subdivision with the face $F$ corresponding to $H$, hence is still a localized marked subdivision. 

If $C$ is a chamber in $\Sigma_{GKZ}(Q)$, let $C/H$ denote the corresponding chamber in $\Sigma_{GKZ}(Q_{/H})$. Different chambers $C$ can result in the same quotient chamber $C/H$. 

Let $X_{C/H}$ denote the phase of $C/H$ in Coulomb problem $Q_{/H}$, and let $Z_{H/H}$ denote the minimal phase in the Coulomb-Higgs problem $Q_{H/H}$, which actually is a point. 

\begin{lemma}\label{l:B-pullback}
Let $C$ be a non-empty GKZ chamber and $H$ a relevant subspace, then 
$$ [X_C: Z_H] = [X_{C/H}: Z_{H/H}].$$
\end{lemma}
\begin{proof}
We induct on $r = \codim_\R H$. If $r=0$, $H = \Lv_\R$. Since by assumption $H$ is relevant, $Z_H=X_{min} \neq \emptyset$. We have
$ [X_C: Z_H]= [ X_C: X_{min}] = 1.$ And we also have 
 $X_{C/H}  = pt$, hence $[X_{C/H}: Z_{H/H}] = 1$. 

Suppose the statement holds when $r < n$, and consider the case $r=n$. For any chamber $C$, consider ``monotone decreasing run'' $C=C_0 \leadsto C_1 \leadsto  \cdots \leadsto C_m$, where $C_m$ is a minimal phase, and let $W_i$ denote the wall separating $C_{i-1}$ and $C_i$. For any wall $W$, write $W\parallel H$ if $\span_{\R}(W)\supset H$.

For any wall $W$ with $W \not\parallel H$, we have $[Z_W: Z_H]=0$. Thus
$$ [X_C: Z_H] = \sum_{W} d_{W} [Z_{W}: Z_H] =  \sum_{W  \parallel H} d_{W} [Z_{W}: Z_H], $$
where summation of $W$ runs over all the wall crossing in the run i.e. $W=W_i$ for some $i$.

For $W$ with $W \parallel H$,  $[Z_{W}: Z_H]$ is computed in the Higgs problem $Q_{\span_\R W}$. By our induction hypothesis, we have $[Z_{W}: Z_H] = [Z_{W/H}, Z_{H/H}]$.  Hence 
$$ [X_C: Z_H] = \sum_{W \parallel H} d_{W} [Z_{W/H}: Z_{H/H}]. $$

On the other hand, the sequence of ``upstair chambers'' in $L_\R^\vee$ is sent to a sequence of ``downstair chambers'' in $L_\R^\vee/H$ (possibly with repetitions). In particular, if $W \not \parallel H$, then the two chambers separated by $W$ are sent to the same chamber in $\Lv_\R/H$. Hence the downstairs wall crossing corresponds to those upstairs wall crossing with $W \parallel H$. We get
$$ [X_{C/H}: Z_{H/H}] = \sum_{W \parallel H} d_{W/H} [Z_{W/H}: Z_{H/H}].   $$
Hence we only need to prove $d_W = d_{W/H}$, and this follows from  $\pi_H^*(\lambda_{W/H}) = \lambda_W$ (up to sign ambiguity of $\lambda_W$). 
\end{proof}

Since a GKZ chamber $C$ corresponds to a coherent stacky fan $\bSigma$, and a relevant subspace $H$ corresponds to a minimal face $F$, we may also denote $X_\bSigma = X_C$ and $Z_{/F} = Z_H$, and denote $n^B_{\bSigma, /F} = n^B_{C,H}$. Let $\Ga = [N]_F$ denote set of $a_i$ in the minimal face. Then, the above statement is 
$$ [X_{\bSigma}: Z_{/F}] = [X_{\bSigma \cap F}: Z_{F/F}] $$

%\section{Transition volume and intersection multiplicity}

\subsection{Recursive Formula for Multiplicity} \label{ss: recursive}
Let $\fcal$ be the set of minimal faces indexing the SOD components. For a variety (or smooth DM stack) $Y$, we use notation
$$ K_0(Y)= K_0(\Coh(Y)) \otimes_\Z \Q. $$
For $F$ a minimal face, we have $$\rank(K_0(Z_{/F})) = \vol(\bSigma^{min}(A_{/F})) = i(\Ga) u(\Ga),$$ 
where $i(\Ga)$ and $u(\Ga)$ are define in \cite{GKZ-book} or \cite{horja2022discriminants}. (See also Section \ref{s:volume-of-ZF}). 

\begin{proposition}\label{p:recursive}
Let $C$ be a GKZ chamber, and $\bSigma$ the corresponding stacky fan. For each minimal face $F \in \fcal$, let $n_F:=[X_\bSigma: Z_{/F}]$ denote the SOD multiplicity of $\Coh(Z_{/F})$ in $\Coh(X_\bSigma)$. Then we have a system of linear equations labelled also by minimal faces in $\fcal$, 
\begin{equation}
\rank(K_0(X_{\bSigma \cap F})) = \sum_{F' \leq F} n_{F'} \rank(K_0(Z_{F/F'})). \label{e:recursive-rank}
\end{equation}
where equation the labelled by $F$ only involves $n_{F'}$ with face $F' \leq F$, and $\rank([Z_{F/F}])=1$. 
\end{proposition}
\begin{proof}
% We define class $[X_\bSigma] \in \Z\fcal$ as
% $$ [X_\bSigma] = \oplus_{F \in \fcal} [X_{\bSigma}: Z_{/F}] [Z_{/F}]. $$
For each minimal face $F$, we have
\begin{equation}
    K_0(X_{\bSigma \cap F}) = \oplus_{F' \leq F} K_0(Z_{F/F'})^{\oplus [X_{\bSigma \cap F}: Z_{F/F'}] }. \label{e:recursive-sod}
\end{equation}  
By our main Lemma (\ref{l:B-pullback}), we have that
$$ [X_{\bSigma \cap F}: Z_{F/F'}] = [X_{\bSigma \cap F \cap F'}: Z_{F'/F'}] = [X_{\bSigma \cap F'}: Z_{F'/F'}] = n_{F'}. $$
%Denote $n_{F'} = [X_{\bSigma}, Z_{/F'}]$ as the unknown variables. 
Taking rank on both sides in Eq \eqref{e:recursive-sod}, we get the desired proposition.
\end{proof}

%We can figure out $n^B_{\bSigma, F}$ inductively. 
%If $F$ is any minimal face, let $\Ga = [N]_{F}$, we may define restriction of the stacky fan $\bSigma$ to that face $F'$ as $\bSigma^{F} = (\NN_\Ga, \Sigma \cap F, S \cap F)$. 

%Apply the above volume relation, we get

\subsection{Examples}
Here we illustrate how to get the SOD multiplicity of $n^B_{\bSigma,F}$ in two different ways. The first way is by wall crossing, and count how each wall $Z_W$ decompose into various $Z_{/F}$. The second way is to use the recursive formula we had. 
\begin{example}

Consider the following $a_i$ points in $\NN=\Z^2$. Our initial stacky fan is $\bSigma = \conv(\{0\} \cup A)$, shown in yellow in the first figure. We have 
$$ X_{\bSigma} = [\C/\Z_3] \times [\C/\Z_2].$$

Then, we show a sequence of shrinking $\bSigma$ (shown in yellow), and show how the lost volume during circuit transition contribute to various SOD components. The total volume of $\bSigma$ is 6. There are 4 minimal faces $$F_1 = \Pi, \quad F_2 = \conv(0, a), \quad F_3=\conv(0,b), \quad F_4 = \{(0,0)\} $$ where $a=(3,0), b=(0,2)$. They corresponds to $$ Z_{/F_1}=pt, \quad Z_{/F_2}=\C, \quad Z_{/F_3}=\C, \quad Z_{/F_4} = \C^2 $$
They all have $\rank(Z_{/F_i})=1$. 

The sequence of 6 figures (from left to right, from top to bottom) represents the sequence of chambers separated by wall-crossings. From figure (1) to (2), we lose volume (blue) of 1 unit in the interior, hence this volume is attributed to $F_1$. From (2) to (3), we lose volume (green) of 1 unit, which is attributed to face $F_2$. From (3) to (4), we lose 1 unit of volume attributed to $F_2$ again. From (4) to (5), we lose 1 unit of volume (cyan) to face $F_3$. From (5) to (6), we lose 1 unit of volume to face $F_1$ again. And finally, from figure (6) to nothing, we lose 1 unit of volume (yellow), which belongs to $F_4$. 
In total,  we have the decomposition of total volume as 
$$ 6 = \underbrace{2}_{F_1} + \underbrace{2}_{F_2}  + \underbrace{1}_{F_3}  + \underbrace{1}_{F_4}. $$

\begin{figure}[hhh]
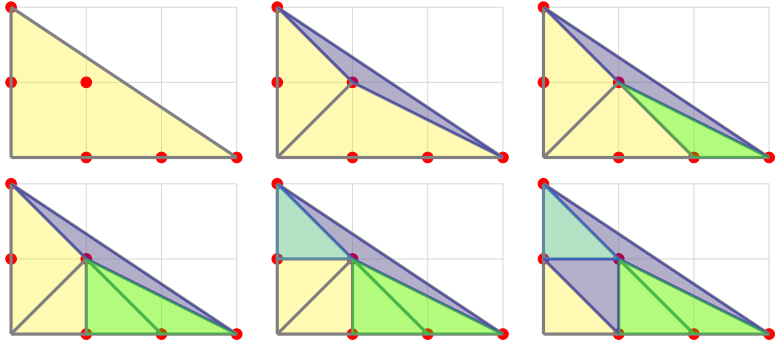

    \centering
$$
\tikz {
\draw [gray, opacity=0.3] (0, 0) grid (3, 2); 
\draw [fill, yellow, opacity = 0.3] (0,0) -- (3, 0) -- (0, 2) -- cycle; 
\filldraw [red] (1,0) circle (2pt);
\filldraw [red] (2,0) circle (2pt);
\filldraw [red] (1,1) circle (2pt);
\filldraw [red] (3,0) circle (2pt);
\filldraw [red] (0,1) circle (2pt);
\filldraw [red] (0,2) circle (2pt);
\draw[gray, very thick] (0,0) -- (0,2);
\draw[gray, very thick] (0,0) -- (3,0);
\draw[gray, very thick] (0,2) -- (3,0);
}\quad
\tikz{
\draw [gray, opacity=0.3] (0, 0) grid (3, 2); 
\draw [fill, yellow, opacity = 0.3] (0,0) -- (3, 0) -- (0, 2) -- cycle; 
\filldraw [red] (1,0) circle (2pt);
\filldraw [red] (2,0) circle (2pt);
\filldraw [red] (1,1) circle (2pt);
\filldraw [red] (3,0) circle (2pt);
\filldraw [red] (0,1) circle (2pt);
\filldraw [red] (0,2) circle (2pt);
\draw[gray, very thick] (0,0) -- (0,2);
\draw[gray, very thick] (0,0) -- (3,0);
\draw[gray, very thick] (0,2) -- (3,0);
\draw[gray, very thick] (0,0) -- (1,1);
\draw[gray, very thick] (1,1) -- (3,0);
\draw[gray, very thick] (1,1) -- (0,2);
\filldraw [blue, opacity=0.3] (1,1)--(3,0) -- (0,2)-- cycle;
}\quad
\tikz{
\draw [gray, opacity=0.3] (0, 0) grid (3, 2); 
\draw [fill, yellow, opacity = 0.3] (0,0) -- (3, 0) -- (0, 2) -- cycle; 
\filldraw [red] (1,0) circle (2pt);
\filldraw [red] (2,0) circle (2pt);
\filldraw [red] (1,1) circle (2pt);
\filldraw [red] (3,0) circle (2pt);
\filldraw [red] (0,1) circle (2pt);
\filldraw [red] (0,2) circle (2pt);
\draw[gray, very thick] (0,0) -- (0,2);
\draw[gray, very thick] (0,0) -- (3,0);
\draw[gray, very thick] (0,2) -- (3,0);
\draw[gray, very thick] (0,0) -- (1,1);
\draw[gray, very thick] (1,1) -- (2,0);
\draw[gray, very thick] (1,1) -- (3,0);
\draw[gray, very thick] (1,1) -- (0,2);
\filldraw [blue, opacity=0.3] (1,1)--(3,0) -- (0,2)-- cycle;
\filldraw [green, opacity=0.3] (1,1)--(3,0) -- (2,0)-- cycle;
}
$$
$$
\tikz{
\draw [gray, opacity=0.3] (0, 0) grid (3, 2); 
\draw [fill, yellow, opacity = 0.3] (0,0) -- (3, 0) -- (0, 2) -- cycle; 
\filldraw [red] (1,0) circle (2pt);
\filldraw [red] (2,0) circle (2pt);
\filldraw [red] (1,1) circle (2pt);
\filldraw [red] (3,0) circle (2pt);
\filldraw [red] (0,1) circle (2pt);
\filldraw [red] (0,2) circle (2pt);
\draw[gray, very thick] (0,0) -- (0,2);
\draw[gray, very thick] (0,0) -- (3,0);
\draw[gray, very thick] (0,2) -- (3,0);
\draw[gray, very thick] (0,0) -- (1,1);
\draw[gray, very thick] (1,0) -- (1,1);
\draw[gray, very thick] (1,1) -- (2,0);
\draw[gray, very thick] (1,1) -- (3,0);
\draw[gray, very thick] (1,1) -- (0,2);
\filldraw [blue, opacity=0.3] (1,1)--(3,0) -- (0,2)-- cycle;
\filldraw [green, opacity=0.3] (1,1)--(3,0) -- (2,0)-- cycle;
\filldraw [green, opacity=0.3] (1,1)--(2,0) -- (1,0)-- cycle;
}\quad
\tikz{
\draw [gray, opacity=0.3] (0, 0) grid (3, 2); 
\draw [fill, yellow, opacity = 0.3] (0,0) -- (3, 0) -- (0, 2) -- cycle; 
\filldraw [red] (1,0) circle (2pt);
\filldraw [red] (2,0) circle (2pt);
\filldraw [red] (1,1) circle (2pt);
\filldraw [red] (3,0) circle (2pt);
\filldraw [red] (0,1) circle (2pt);
\filldraw [red] (0,2) circle (2pt);
\draw[gray, very thick] (0,0) -- (0,2);
\draw[gray, very thick] (0,0) -- (3,0);
\draw[gray, very thick] (0,2) -- (3,0);
\draw[gray, very thick] (0,0) -- (1,1);
\draw[gray, very thick] (1,0) -- (1,1);
\draw[gray, very thick] (0,1) -- (1,1);
\draw[gray, very thick] (1,1) -- (2,0);
\draw[gray, very thick] (1,1) -- (3,0);
\draw[gray, very thick] (1,1) -- (0,2);
\filldraw [blue, opacity=0.3] (1,1)--(3,0) -- (0,2)-- cycle;
\filldraw [green, opacity=0.3] (1,1)--(3,0) -- (2,0)-- cycle;
\filldraw [green, opacity=0.3] (1,1)--(2,0) -- (1,0)-- cycle;
\filldraw [cyan, opacity=0.3] (1,1)--(0,2) -- (0,1)-- cycle;
}\quad
\tikz{
\draw [gray, opacity=0.3] (0, 0) grid (3, 2); 
\draw [fill, yellow, opacity = 0.3] (0,0) -- (3, 0) -- (0, 2) -- cycle; 
\filldraw [red] (1,0) circle (2pt);
\filldraw [red] (2,0) circle (2pt);
\filldraw [red] (1,1) circle (2pt);
\filldraw [red] (3,0) circle (2pt);
\filldraw [red] (0,1) circle (2pt);
\filldraw [red] (0,2) circle (2pt);
\draw[gray, very thick] (0,0) -- (0,2);
\draw[gray, very thick] (0,0) -- (3,0);
\draw[gray, very thick] (0,2) -- (3,0);
\draw[gray, very thick] (1,0) -- (0,1);
\draw[gray, very thick] (1,0) -- (1,1);
\draw[gray, very thick] (0,1) -- (1,1);
\draw[gray, very thick] (1,1) -- (2,0);
\draw[gray, very thick] (1,1) -- (3,0);
\draw[gray, very thick] (1,1) -- (0,2);
\filldraw [blue, opacity=0.3] (1,1)--(3,0) -- (0,2)-- cycle;
\filldraw [green, opacity=0.3] (1,1)--(3,0) -- (2,0)-- cycle;
\filldraw [green, opacity=0.3] (1,1)--(2,0) -- (1,0)-- cycle;
\filldraw [cyan, opacity=0.3] (1,1)--(0,2) -- (0,1)-- cycle;
\filldraw [blue, opacity=0.3] (1,1)--(1,0) -- (0,1)-- cycle;
}
$$
    \caption{Multiplicities from Volume allocation.}
    \label{fig:my_label}
\end{figure}

Next, we compute the SOD multiplicities again using the recursive relation.
Denote $n_i = [X_\bSigma, Z_{/F_i}]$, and denote $m_i = \rank(Z_{/F_i})$. 
First, since $F_4$ is smallest face, we have $n_4=1$. 
Then, consider the equation for face $F_3$. $F_3$ has a proper subface $F_4$, and we have $m_4=1, m_3=1$. We have
$$ 2 = \rank(X_{\bSigma \cap F_3}) = \rank([\C/\Z_2])= n_3 + n_4 = n_3 + 1$$
hence $n_3=1$. 
Next, we compute $n_2$ using $F_2$ equation, we get
$$ 3 =  \rank(X_{\bSigma \cap F_2}) =  \rank([\C/\Z_3])= n_2+n_4 = n_2 + 1$$
hence $n_2=2$. 
Lastly, we compute $n_1$, using $F_1$, which gives 
$$ 6 = n_1 + n_2 + n_3 + n_4 = n_1 + 2 + 1 + 1 $$
Thus $n_1=2$. 
\end{example}

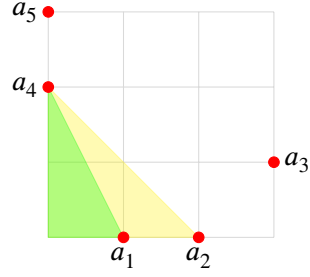
We note that the SOD components really depends on the entire GIT problem, instead of just the initial phase $X_\bSigma$ that we try to decompose. 

\begin{example}
Consider a 1-dimensional problem with rays (in the stacky fan) $\{1,3\}$ and the problem with rays $\{1,2,3\}$. Then, we get two different SODs of $[\mathbb C/\mathbb{Z}_3]$, one is $\langle \C, \C/\Z_2 \rangle$ and the other is  $\langle \C, \C, \C \rangle$. 
\end{example}

\section{Multiplicity Conjecture for toric CY GIT}
Consider a toric GIT problem $Q: \ZvN \to \Lv$ that satisfies the Calabi-Yau condition, i.e.  $\sum_{i=1}^N Q(e_i^\vee)=0$.

In this case, GKZ defined a principal $A$-determinant $E_A$, an integer polynomial with $N$ variables. $E_A$ has many nice properties, its Newton polytope equals the secondary polytope, and its zero-loci descent to a hypersurface $\nabla_{GKZ} \In \Lv_\CS$, the GKZ discriminant loci. 

It is known that divisor $\nabla_{GKZ}$ has decomposition into components $\nabla_\Ga$, 
$$ \nabla_{GKZ} = \sum_{\Ga \in \fcal } m_\Ga \nabla_\Ga, $$
where $\fcal$ denote the set of minimal faces, where $m_\Ga = \rank(K_0(Z_\Ga))$ are certain multiplicities. 

It is also known that the tropicalization of $\nabla_{GKZ}$ is the union of walls in $\Sigma_{GKZ}$. We will now decorate the walls by multiplicities in free abelian group generated by minimal faces.

Let $\Z \fcal$ denote the free abelian group generated by $\fcal$, with basis denoted by $[\Ga]$. Then for each wall $W$, we define
$$ [W] := \sum_\Ga [Z_W: Z_\Ga] \cdot [\Ga] = \sum_\Ga n^B_{W, \Ga} \cdot [\Ga]. $$

Let $\wcal$ denote the set of walls in $\Sigma_{GKZ}(Q)$. Let $\Z \wcal$ be the free abelian group generated by $\wcal$, with basis denoted as $\la W \ra$. Then any tropical complex supported on $\Sigma_{GKZ}$ is equivalent to a non-negative linear combination of the walls. In particular, we have 
$$ \nabla_\Ga^{trop} := \sum_{W \in \wcal} n^A_{W, \Ga} \la W \ra $$

\begin{definition}
Let $Q$ be a toric CY GIT problem. A tropical complex is an element in $\Z \wcal$, and a decorated tropical complex is an element in $\Z(\wcal \times \fcal)$. 
We define the A-model tropical complexes as
$$ \nabla^A_{GKZ}= \sum_{\Ga} \nabla_\Ga^{trop} \cdot [\Ga] = \sum_{\Ga, W} n^A_{W, \Ga} \la W \ra [\Ga], \quad \nabla_\Ga^A = \sum_W n^A_{W, \Ga} \la W \ra  $$
We define the B-model tropical complexes as
$$ \nabla^B_{GKZ} = \sum_{W} \la W \ra [W] = \sum_{W, \Ga}  n^B_{W, \Ga} \la W \ra [\Ga], \quad \nabla_\Ga^B = \sum_W n^B_{W, \Ga} \la W \ra  $$

We write $\nabla^A_{GKZ}(Q), \nabla^A_\Ga(Q), \nabla^B_{GKZ}(Q), \cdots$ to emphasize the dependence on $Q$. 
\end{definition}

Our goal is to show that $\nabla^A(Q) = \nabla^B(Q)$, or more concretely for each $\Ga$, $\nabla^A_\Ga(Q) = \nabla^B_\Ga(Q)$. We first reduce the task to showing just the top dimensional minimal face. 

Let $Q_\Ga: (\Zv)^\Ga \to \Lv_\Ga$ denote the Coulomb problem for minimal face $\Ga$. Note that $Q_\Ga$ still defines a toric CY GIT. Let $\pi_\Ga: \Lv_\R \to (\Lv_\Ga)_\R$. Define $\pi_\Ga^*: \Z(\wcal_\Ga \times \fcal_\Ga) \to \Z(\wcal \times \fcal)$. For a wall $W_\Ga \in \wcal_\Ga$, we define $\pi_\Ga^*(\la W_\Ga\ra) = \la \pi^{-1} W_\Ga \ra = \sum_{W \in \wcal, \pi_\Ga(W) = W_\Ga} \la W \ra$. For any $\Ga' \in \fcal_\Ga$,  since $\fcal_\Ga \In \fcal$, we have $\pi_\Ga^*([\Ga']) = [\Ga']$ . 

\begin{proposition}\label{p:pullback2}
$ \nabla^A_\Ga(Q) = \pi_\Ga^*(\nabla^A_\Ga(Q_\Ga))$ and 
$ \nabla^B_\Ga(Q) = \pi_\Ga^*(\nabla^B_\Ga(Q_\Ga))$.
\end{proposition}
\begin{proof}
On the A-side, this follows from relation from the complex hypersurfaces $\nabla_\Ga(Q) = \pi_\Ga^{-1} (\nabla_\Ga(Q_\Ga) )$. 

On the B-side, we know from Lemma \ref{l:B-pullback} that $[Z_W: Z_\Ga] = [Z_{W_\Ga}: Z_{\Ga / \Ga}]$. Thus
$$ \nabla_\Ga^B(Q) = \sum_W [Z_W: Z_\Ga] \la W \ra = \sum_{W: W \parallel \Ga} [Z_{W_\Ga}: Z_{\Ga / \Ga}] \la W \ra = \sum_{W_\Ga \in \wcal_\Ga} [Z_{W_\Ga}: Z_{\Ga/\Ga}] \pi_\Ga^*(\la W_\Ga \ra) =\pi_\Ga^*(\nabla^B_\Ga(Q_\Ga)).$$
\end{proof}

%Hence, it suffices to prove that for CY toric GIT problem $Q$, we have $\nabla^A_\Pi(Q) = \nabla^B_\Pi(Q)$. (If $\Pi$ is not a minimal face, then both sides are zero, and there is nothing to prove.) 

We define a rank map, 
$$\rank: \Z \fcal \to \Z, \quad \rank([\Ga]):= \rank (K_0(Z_\Ga)). $$ And we extend it by linearity to 
$$ \rank : \Z(\wcal \times \fcal) \to \Z(\wcal). $$

Horja-Katzarkov proved a numerical version of the desired theorem.
\begin{proposition}[\cite{horja2022discriminants}, Theorem 3.5] \label{p:rank}
$ \rank(\nabla^A_{GKZ}) = \rank(\nabla^B_{GKZ})$
\end{proposition} 
\begin{proof}
In loc.cit, we have 
$\rank([W]) = \sum_\Ga n^A_{W, \Ga} \rank([\Ga])$. 
Hence
$$ \rank(\nabla^B_{GKZ}) = \sum_W  \rank([W])\la W \ra = \sum_{W, \Ga} n^A_{W, \Ga} \rank([\Ga]) \la W \ra  = \rank(\nabla^A_{GKZ}). $$
\end{proof}

Now we are ready to prove our main theorem. 
\begin{theo} \label{t: main}
For any toric CY GIT problem $Q$, we have $\nabla^A_{GKZ}(Q) = \nabla^B_{GKZ}(Q)$, i.e. $n^A_{W, \Ga} = n^B_{W, \Ga}$ for any wall $W$ and minimal faces $\Ga$.  
\end{theo}
\begin{proof}
We induct on the rank of $L$. In the case $\rank L = 0$, there is no wall, hence nothing to prove. Assume the case for $\rank L < n$ is proven, and we have $\rank L =n$.

Let $\Ga_{0}$ be largest minimal face, concretely $\Ga_{0} = \{i: q_i \neq 0 \}$.  For any minimal face $\Ga \neq \Ga_0$, we have $\rank L_\Ga < n$, hence by induction $\nabla_\Ga^A(Q_\Ga) = \nabla_\Ga^B(Q_\Ga)$. By Proposition \ref{p:pullback2}, we have $\nabla_\Ga^A(Q) =\nabla_\Ga^B(Q)$ for all minimal faces $\Ga \neq \Ga_0$. 
Thus, we have 
$$ \nabla^A_{GKZ}(Q) - \nabla^B_{GKZ}(Q) = \sum_\Ga (\nabla^A_\Ga(Q) - \nabla^B_\Ga(Q)) [\Ga] = (\nabla^A_{\Ga_0}(Q) - \nabla^B_{\Ga_0}(Q)) [\Ga_0]$$
Now, take $\rank$ on both sides, using $\rank([\Ga_0]) = \rank(K_0\Coh(Z_{\Ga_0}))=\rank(K_0\Coh(\A^{\Ga_0^c}))=1 \neq 0$, and Horja-Katzarkov result Proposition \ref{p:rank}, we get $$ 0 = \nabla^A_{\Ga_0}(Q) - \nabla^B_{\Ga_0}(Q). $$
Thus $\nabla^A_{GKZ}(Q) = \nabla^B_{GKZ}(Q)$. 
\end{proof}

\subsection{Examples}
\begin{example}
Consider the example in \cite[Example 3.7]{kite-segal} with an additional weight $-\det V=(-1,-2)$ (marked in blue) to get a toric CY GIT problem. There are three relevant subspaces: 
$$ H_0=L^\vee_{\mathbb{R}}, \quad H_1=\span_{\R}(q_3,q_4, q_6), \quad H_2=\{0\} $$ 
We compute the tropical complexes in this CY problem. 

\begin{figure}[hhh]
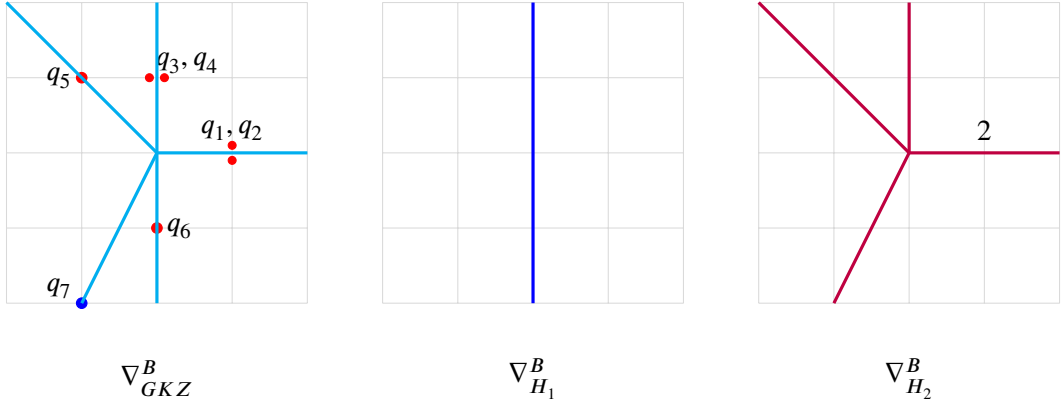

    \centering
\tikz[baseline=-4]{
\begin{scope}
\draw [gray, opacity=0.3] (-2, -2) grid (2, 2); 
\filldraw [red] (1,0.1) circle (1.5pt);
\filldraw [red] (1,-0.1) circle (1.5pt);
\node at (1,0.3) {$q_1, q_2$};
\filldraw [red] (0.1,1) circle (1.5pt);
\filldraw [red] (-0.1,1) circle (1.5pt);
\node at (0.4,1.2) {$q_3, q_4$};
\filldraw [red] (-1,1) circle (2pt);
\node at (-1.3,1) {$q_5$};
\filldraw [red] (0,-1) circle (2pt);
\node at (0.3,-1) {$q_6$};
\filldraw [blue] (-1,-2) circle (2pt);
\node at (-1.3,-1.8) {$q_7$};

\draw[cyan, very thick] (0,0) -- (2,0);
\draw[cyan, very thick] (0,0) -- (0,2);
\draw[cyan, very thick] (0,0) -- (-2,2);
\draw[cyan, very thick] (0,0) -- (-1,-2);
\draw[cyan, very thick] (0,0) -- (0,-2);
\node at (0, -3) {$\nabla^{B}_{GKZ}$};
\end{scope}
\begin{scope}[shift={(5,0)}]
\draw [gray, opacity=0.3] (-2, -2) grid (2, 2); 
\draw[blue, very thick] (0,-2) -- (0,2);
\node at (0, -3) {$\nabla^{B}_{H_1}$}; 
\end{scope} 
\begin{scope}[shift={(10,0)}]
\draw [gray, opacity=0.3] (-2, -2) grid (2, 2); 
\draw[purple, very thick] (0,0) -- (2,0);
\draw[purple, very thick] (0,0) -- (0,2);
\draw[purple, very thick] (0,0) -- (-2,2);
\draw[purple, very thick] (0,0) -- (-1,-2);
\node at (1,0.3) {$2$};
\node at (0, -3) {$\nabla^{B}_{H_2}$}; 
\end{scope}

}

    \caption{Decomposition of $\nabla^B_{GKZ} = \sum_H \nabla^B_H$, where $H$ labels the relevant subspaces.}
    \label{fig:}
\end{figure}

($H_0$) Since the problem is CY, every chamber is a minimal chamber. The Coulomb problem is the trivial problem $\Z^0\xrightarrow{0}\Z^0$, and $\nabla_{H_0}^B=\emptyset$.

($H_1$) From the Higgs problem $\Gamma_1^c=\{1,3,4\}$ we know $Z_{H_1}=\A^1$. The Coulomb problem for $\Gamma_1^c$ is the Atiyah flop $Q=(1,1,-1,-1)$. We claim $\nabla_{H_1}^B$ has multiplicity 1 on the contributing walls, which is a result proven in \cite{kite-segal}  when $\rank L=1$. (Indeed, any irreducible hypersurface in $\C^*$ is a point with multiplicity 1). 

($H_2$) In this case, the Coulomb problem is the entire GIT problem. The picture on the right is the tropical complex $\nabla_{H_2}^B$ where one of the walls has multiplicity 2. These multiplicities can be computed by looking at the Higgs problem. For the wall with multiplicity 2, $Z_W=\P^1$ and $Z_\Gamma=pt$, and $\Coh(\P^1)=\langle \Coh(pt), \Coh(pt)\rangle$. For other walls, $Z_W=pt$ and $Z_\Gamma=pt$.
\end{example}

\begin{example}
% This example shows that the conjectural formula \cite[Remark 4.7] {halpern2016autoequivalences} is false. 

Consider toric CY GIT problem $Q: \Z^5 \to \Z^3$, with weight matrix 
$$ Q = \begin{pmatrix}
1 & 1 & 0 & 0  & -n  & n-2 \\
0 & 0 & 1 & 0 & 1 & -2 \\
0 & 0  & 0 & 1 & 0 &  -1
\end{pmatrix}
$$
and let $q_i$ be the column vectors. 
Assume $n \geq 2$. Consider the $xy$ plane, which contains weights $q_1=q_2=(1,0), q_3=(0,1), q_5=(-n,1)$ (omitting the 3rd coordinate). The following is a picture of the $xy$ plane for $n=2$. 
$$
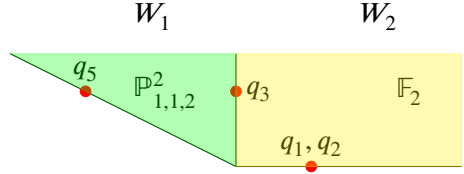
\tikz{
\draw (0,0) -- (3,0);
\draw (0,0) -- (0,1.5);
\draw (0,0) -- (-3,1.5);
\filldraw [red] (1,0) circle (2pt);
\filldraw [red] (0,1) circle (2pt);
\filldraw [red] (-2,1) circle (2pt);
\node at (1,0) [above] {$q_1,q_2$};
\node at (0,1) [right] {$q_3$};
\node at (-2,1)[above] {$q_5$};
\node at (2,1) [right] {$\F_2$};
\node at (-1.5,1) [right] {$\P^2_{1,1,2}$};
\node at (-1.5, 2) [right] {$ W_1$};
\node at (1.5, 2) [right] {$ W_2$};
\draw [fill, yellow, opacity=0.3] (0,0) rectangle (3,1.5);
\draw [fill, green, opacity=0.3] (0,0) -- (0,1.5) -- (-3,1.5) -- cycle;
}
$$
There are two phases on this plane, corresponding to weighted projective space $Z_{W_1}=\P^2_{1,1,n}$ ($W_1=\cone(q_3,q_5)$) and Hirzebruch surface $Z_{W_2}=\F_n$ (for $W_2=\cone(q_1, q_3)$). We have SOD multiplicities (see also the example in the introduction of \cite{BFDKK})
$$ [\P^2_{1,1,n}] = (n+2) [pt], \quad [\F_n] = 4 [pt]. $$
This means a local curve transverse to wall $W_1$ will intersect $\nabla_{GKZ}$ at $n+2$ points, and a local curve transverse to wall $W_2$ will intersect $\nabla_{GKZ}$ at $4$ points. \footnote{Our calculation seems to be different from the formula suggested in Remark 4.7 of \cite{halpern2016autoequivalences}, which always predicts $4$ no matter which wall the curve intersects. Their conjecture that the total A-side multiplicity equals B-side length of full exceptional collection on the wall is true and is proven by \cite{horja2022discriminants}. }
% This agrees with the expected size of the exceptional collection. 

One can also verify that $\nabla^{B}_{GKZ}$ here is a balanced tropical complex. Consider a local tropical curve $C$ with tangent direction $(0,0,1)$ intersects $\nabla^{B}_{GKZ}$ near point $q_3$. If $C$ intersects at $W_1$, then we get  intersection multiplicity 
$$ C \cdot \nabla^{B}_{GKZ} = (C \cdot \la W_1 \ra) \cdot  \rank (Z_{W_1}) = 1 \cdot (n+2) = n+2, $$
where tropical intersection number $C \cdot W$ is defined as the unsigned volume $|\det (v_1 \wedge v_2 \wedge v_3)|$, where $v_i \in \Z^3$ and $v_1$ is primitive generator of $C$ and $v_2, v_3$ are generators of the sublattice $\span_\R(W) \cap \Z^3$.  
If we move the curve to the right, $C$ intersects with $W_2$ and an additional wall $W_3=\cone(q_3, q_6)$, and we get 
$$ C \cdot \nabla_{GKZ}^{B} = (C \cdot W_2) \rank(Z_{W_2}) + (C \cdot W_3)  \rank(Z_{W_3}) = 1 \cdot 4 + (n-2) \cdot 1 = n+2, $$ 
where $Z_{W_3} = pt$. 

\end{example}

\appendix
\section{Adaptation of GKZ's setup}
In the paper, we quoted results from \cite{GKZ-book} whereas the setup do not quite match. In the appendix, we show that with minor modification, the generalized case reduces to the original GKZ setup, hence the GKZ conclusion still applies. 
In GKZ, one start with a finite set $A \In \Z^{k-1}$, where as in our setup, we can have a map $A: \Z^N \to \NN$, where $\NN$ can have torsion and $a_i=A(e_i)$ can coincide with each other. Here in the appendix we show that the difference in the setup is minor. 

First, we generalize $A$ to be a ``multi-set'' while keeping $\NN$ a lattice. Suppose we have $A: \Z^N \to \Z^{k-1} \times \Z$, and let $A'$ denote the image set $\{A(e_i)\}$, then we have a map $s: [N] \onto A'$. We claim that the GKZ discriminant loci for $A$ is a pull-back of that for $A'$, where the fiber is a product of pair-of-pants. Indeed, consider the example,
$$ A: \Z^3 \to \Z, \quad e_i \mapsto 1 $$
then the corresponding $W = (c_1 + c_2 + c_3) z$, and the discrimiants happens exactly at $c_1+c_2+c_3=0$, which after quotient by $\C^*$ become the pair-of-pants in $\Lv_\CS = (\C^*)^2$. 

Next, we generalize the case where $\NN$ have torsion. Suppose
$Q: (\Z^N)^\vee \to \Lv$ has finite cokernel. Define $\h L^\vee = im(Q) \In \Lv$, then we have 
\be \label{e:LvZ} 0 \to \h L^\vee \to \Lv \to \MM_1 \to 0. \ee
Dualize, we get 
\be \label{e:LZ} 0 \to L \to \h L \to \NN_{tors} \to 0 \ee
where we used that $\NN_{tors} = \Ext^1(\MM_1, \Z)$. In other words, $\h L$ is the saturation of $L$ in $\Z^N$. 
%Hence 
%\be \label{e:hLN} 0 \to \h L \to \Z^N \to \NN_{free} \to 0 \ee
%and $\Hom(-, \C^*)$, we get
\be \label{e:hLCS} 0 \to (\MM_0)_\CS \to (C^*)^N \to \h L_\CS^\vee \to 1. \ee
Apply $\Hom(-, \C^*)$ to Eq \eqref{e:LvZ}, 
\be \label{e:LhLvCS} 0 \to  \MM_1 \to \h L^\vee_\CS \to \Lv_\CS \to 1 \ee
Thus, although the fiber of $(C^*)^N \to L_\CS^\vee$ has $|\MM_1|$ disconnected component of $(\MM_0)_\CS$, and the point $b \in L_\CS^\vee$ is ``bad'' if the function $W$ restrict any component is singular (at finite place or at infinity), we can first work with $\h Q: (\Z^N)^\vee \to \h L^\vee$, and get $\nabla_{GKZ}(\h Q) \In \h L_\CS^\vee$, then  pushforward it along $\h L_\CS^\vee \to \Lv_\CS$ to get $\nabla_{GKZ}(\h Q)$.

\bibliographystyle{amsalpha}
\bibliography{HMS}{}

\end{document}